\documentclass[11pt]{article}
\usepackage[margin=3cm]{geometry}
\usepackage{amsmath,amsthm,amsfonts,amssymb,amscd}
\usepackage{tikz}

\usepackage{dsfont}
\usepackage{framed}
\usepackage[utf8]{inputenc}
\usepackage{mathtools}
\usepackage{multicol}
\usepackage{multirow}
\usepackage{enumerate}
\usepackage{bbm}

 \usepackage[colorlinks = true]{hyperref}
 \hypersetup{
 	linkcolor = blue,
 	urlcolor = blue,
 	citecolor = blue
 }

\newcommand{\reff}[1]{(\ref{#1})}

\newcommand{\IH}{\mathbb{H}}
\newcommand{\ID}{\mathbb{D}}

\newcommand{\IE}{\mathbb{E}}
\newcommand{\IK}{\mathbb{K}}

\newcommand{\sG}{\mathcal{G}}
\newcommand{\sA}{\mathcal{A}}

\newcommand{\sX}{\mathcal{X}}

\newcommand{\dto}{\overset{d}{\to}}
\newcommand{\R}{\mathbb{R}}

\newcommand{\Z}{\mathbb{Z}}
\newcommand{\N}{\mathbb{N}}
\newcommand{\E}{\mathbb{E}}

\newcommand{\IP}{\mathbb{P}}

\newcommand{\sF}{\mathcal{F}}

\newcommand{\sH}{\mathcal{H}}
\newcommand{\G}{\mathbb{G}}
\newcommand{\Ii}{\mathbbm{1}}
\newcommand{\Cov}{\mathrm{Cov}}
\newcommand{\Var}{\mathrm{Var}}

\newcommand{\eps}{\varepsilon}

\newcommand{\norm}[1]{\left\lVert#1\right\rVert}

\makeatletter
\newcommand{\undersim}[1]{\mathrel{\mathpalette\@undersim{#1}}}
\newcommand{\@undersim}[2]{%
  \vcenter{%
    \ialign{%
      ##\cr
      $\m@th#1#2$\cr
      \noalign{\nointerlineskip\kern.2ex}
      $\m@th#1\sim$\cr
      \noalign{\kern-.4ex}
    }%
  }%
}

\makeatother

\newtheorem{thm}{Theorem}[section]
\newtheorem{lem}[thm]{Lemma}
\newtheorem{cor}[thm]{Corollary}

\newtheorem{assum}[thm]{Assumption}

\theoremstyle{definition}
\newtheorem{defi}[thm]{Definition}
\newtheorem{rem}[thm]{Remark}

\numberwithin{equation}{section}

\RequirePackage[numbers]{natbib}

\begin{document}

\title{Empirical process theory for nonsmooth functions under functional dependence}

\thispagestyle{empty}

\begin{center}
{\LARGE \bf Empirical process theory for nonsmooth functions under functional dependence}\\
{\large Nathawut Phandoidaen, Stefan Richter}\\

{phandoidaen@math.uni-heidelberg.de, stefan.richter@iwr.uni-heidelberg.de}\\

{\small Institut für angewandte Mathematik, Im Neuenheimer Feld 205, Universität Heidelberg}\\
\today
\end{center}

\begin{abstract}
    {We provide an empirical process theory for locally stationary processes over nonsmooth function classes. An important novelty over other approaches is the use of the flexible functional dependence measure to quantify dependence. A functional central limit theorem and  nonasymptotic maximal inequalities are provided. The theory is used to prove the functional convergence of the empirical distribution function (EDF) and to derive uniform convergence rates for kernel density estimators both for stationary and locally stationary processes. A comparison with earlier results based on other measures of dependence is carried out.}
\end{abstract}

\section{Introduction}
\label{sec_intro}

Empirical process theory is one of the key concepts in proving uniform convergence rates and weak convergence of composite functionals. It is preferable to have a theory which can be applied to observations which are dependent but also nonstationary. Locally stationary processes allow for a smooth change of the distribution over time but can locally be approximated by stationary processes and thus provide more flexible time series models (cf. \cite{dahlhaus2019}). This paper extends a theory of our recent paper \cite{empproc} where we have established an empirical process theory for locally stationary processes under functional dependence considering function classes that are at least H\"older-continuous. Here, we additionally allow for nonsmooth functions, in particular, our framework includes (but is by far not limited to) the empirical distribution function (EDF).

The only papers that are know to the authors that explicitly deal with functional convergence of locally stationary processes are \cite{empproc} and \cite{mayer2019}. For stationary processes, a vast range of theoretical results are available. A prominent idea to measure dependence of random variables is given by mixing (cf. \cite{doukhan94}). The publications \cite{arcones}, \cite{yu_bin} and \cite{doukhan1995} derive large deviation results and uniform central limit theorems under absolute regularity ($\beta$-mixing). In \cite{rio2013}, refined results are available. Other general theories are based on Markov chains and other types of mixing, cf. the overview in \cite{empproc}.

Regarding the functional weak convergence of the EDF, more specific conditions were derived in the literature for stationary observations. \cite[Theorem 4]{durieu14} provide functional convergence of the EDF using bounds for covariances of H\"older functions of the random variables. Another abstract concept was introduced by \cite{berkes09} via S-mixing (for stationary mixing), which imposes the existence of $m$-dependent approximations of the original observations. They then derive strong approximations and uniform central limit theorems for the EDF. Other approaches were presented in \cite{dehling2001} and \cite{dedecker10}. In \cite{wu2008} and \cite{mayer2019} uniform central limit theorems for the EDF were derived for stationary and piece-wise locally stationary processes under functional dependence.

Our empirical process theory is derived under the dependence concept of functional dependence (cf. \cite{wu2005anotherlook}). In combination with the theory of martingales it allows for sharp large deviation inequalities (cf. \cite{Wu13} or \cite{wuzhang2017}). We assume that $X_i =  (X_{ij})_{j=1,...,d}$, $i = 1,...,n$, is a $d$-dimensional Bernoulli shift process of the form
\begin{equation}
    X_i = J_{i,n}(\sA_i),\label{representation_x}
\end{equation}
where $\sA_i = \sigma(\varepsilon_i,\varepsilon_{i-1},...)$ is the sigma-algebra generated by $\varepsilon_i$, $i \in\Z$, a sequence of i.i.d. random variables in $\R^{\tilde d}$ ($d,\tilde d \in \N$), and some measurable function $J_{i,n}:(\R^{\tilde d})^{\N_0}\to \R$, $i=1,...,n$, $n\in\N$. For a real-valued random variable $W$ and some $\nu > 0$, we define $\|W\|_{\nu} := \IE[|W|^{\nu}]^{1/\nu}$. 
If $\varepsilon_k^{*}$ is an independent copy of $\varepsilon_k$, independent of $\varepsilon_i, i\in\Z$, we define $\sA_i^{*(i-k)} := (\varepsilon_i,...,\varepsilon_{i-k+1},\varepsilon_{i-k}^{*},\varepsilon_{i-k-1},...)$ and set $X_i^{*(i-k)} := J_{i,n}(\sA_{i}^{*(i-k)})$. The uniform functional dependence measure is then given by
\begin{equation}
    \delta_{\nu}^{X}(k) = \sup_{i=1,...,n}\sup_{j=1,...,d}\big\|X_{ij} - X_{ij}^{*(i-k)}\big\|_\nu.\label{definition_uniform_functional_dependence_measure}
\end{equation}
The value $\delta_{\nu}^{X}$ measures the impact of $\varepsilon_0$ on $X_k$. The representation \reff{representation_x} actually does cover a large variety of processes. In \cite{borkar1993} it was motivated that the set of all processes of the form $X_i = J(\varepsilon_i,\varepsilon_{i-1},...)$ should be equal to the set of all stationary and ergodic processes. We additionally allow $J$ to vary with $i$ and $n$ to cover processes which change their stochastic behavior over time. This is exactly the form of the so-called locally stationary processes discussed in \cite{dahlhaus2019}.

Since we are working in a time series context, many applications ask for functions $f$ that not only depend on the actual observation of the process but on the whole (infinite) past $Z_i := (X_i,X_{i-1},X_{i-2},...)$. In the course of this paper, we aim to derive asymptotic properties of the empirical process
\begin{equation}
    \G_n(f) := \frac{1}{\sqrt{n}}\sum_{i=1}^{n}\big\{ f(Z_i,\frac{i}{n}) - \E f(Z_i,\frac{i}{n})\big\}, \quad f \in \sF,\label{definition_empirical_process}
\end{equation}
where
\[
    \sF \subset \{f:(\R^d)^{\N_0}\times [0,1] \to \R \text{ measurable}\}.
\]
Let $\IH(\varepsilon,\sF,\|\cdot\|)$ denote the bracketing entropy, that is, the logarithm of the number of $\varepsilon$-brackets with respect to some distance $\|\cdot\|$ that is necessary to cover $\sF$ (this is made precise at the end of this section). We will define a distance $V_n$ which guarantees weak convergence of \reff{definition_empirical_process} if the corresponding bracketing entropy integral $\int_0^{1}\sqrt{\IH(\varepsilon,\sF,V_n)}d\varepsilon$ is finite.


Our main contributions are the following:
\begin{itemize}
    \item We derive maximal inequalities for $\G_n(f)$ where the class $\sF$ consists of nonsmooth functions.
    \item We state conditions to ensure asymptotic tightness and functional convergence of $\G_n(f)$, $f\in\sF$.
\end{itemize}
Eventhough our theory allows for general function classes, we will have a special focus on the EDF. In particular, we derive functional convergence of the EDF under weak conditions on the moments and the dependence structure of the process $X_i$. We will see that our results typically pose weaker conditions on the underlying dependence structure than comparable results for the stationary case mentioned above. In particular, we compare our results with \cite{mayer2019} where the authors discussed the EDF of piece-wise locally stationary processes.

The paper is structured as follows. In Section \ref{sec_model}, we present our main result Theorem \ref{cor_functional_central_limit_theorem_martingale}, the functional central limit theorem under minimal moment conditions. We then derive a version for stationary processes, and discuss its application on empirical distribution functions where the underlying process is either stationary or locally stationary. It is the aim of Section \ref{sec_further_appli} to show a wide range of applicability of our theory. Some assumptions are postponed to Section \ref{sec_clt}, where a new multivariate central limit theorem for locally stationary processes is presented. In Section \ref{sec_max_ineq_tight} we provide new maximal inequalities for $\G_n(f)$ in case of a finite and infinite function class $\sF$. In Section \ref{sec_conclusion} a conclusion is drawn. We postpone all detailed proofs to the Supplementary Material \ref{suppA}.

We now introduce some basic notation. For $a,b\in\R$, let $a\wedge b := \min\{a,b\}$, $a\vee b := \max\{a,b\}$. For $k\in\N$,
\begin{equation}
    H(k) := 1 \vee \log(k) \label{H_var}
\end{equation}
which naturally appears in large deviation inequalities. For a given finite class $\sF$, let $|\sF|$ denote its cardinality. We use the abbreviation
\begin{equation}
    H = H(|\sF|) = 1 \vee \log |\sF| \label{H_fix}
\end{equation}
if no confusion arises. For some distance $\|\cdot\|$, let $\N(\varepsilon,\sF,\|\cdot\|)$ denote the bracketing numbers, that is, the smallest number of $\varepsilon$-brackets $[l_j,u_j] := \{f\in \sF: l_j \le f \le u_j\}$ (i.e. measurable functions $l_j,u_j:(\R^{d})^{\N_0}\times[0,1]\to\R$ with $\|u_j - l_j\| \le \varepsilon$ for all $j$) to cover $\sF$. Let $\IH(\varepsilon,\sF,\|\cdot\|) := \log \N(\varepsilon,\sF,\|\cdot\|)$ denote the bracketing entropy. For $\nu \ge 1$, let
\[
    \|f\|_{\nu,n} := \Big(\frac{1}{n}\sum_{i=1}^{n}\big\|f\big(Z_{i},\frac{i}{n}\big)\big\|_\nu^\nu\Big)^{1/\nu}.
\]

\section{A functional central limit theorem under functional dependence and application to empirical distribution functions}
\label{sec_model}



A process $X_i$, $i=1,...,n$ is called locally stationary if for each $u\in[0,1]$, there exists a stationary process $\tilde X_i(u)$ approximating $X_i$ for $i=1,...,n$, i.e. $X_i \approx \tilde X_i(u)$ if $|u - \frac{i}{n}|$ is small (cf. \cite{dahlhaus2019}). The exact form needed is stated in Assumption \ref{ass_clt_process}. Thus, $X_i$ behaves stationary around each fixed (rescaled) time point $u\in[0,1]$, but over the whole time period $i=1,...,n$ its distribution can change drastically. Deterministic properties of the process like expectation, covariance, spectral density or empirical distribution functions therefore also depend on the rescaled time $u\in[0,1]$. As an example, consider the localized empirical distribution function of $X_i$,
\begin{equation}
    \hat G_{n,h}(x,v) := \frac{1}{nh}\sum_{i=1}^{n}K\big(\frac{i/n-v}{h}\big) \Ii_{\{X_i \le x\}},\label{example_distribution_emp_eq1}
\end{equation}
where $K:\R \to \R$ is a kernel function and $h = h_n > 0$ a bandwidth. The goal of this paper is to provide a general empirical process theory which allows to show, for instance, a functional central limit theorem of $\hat G_{n,h}(x,v)$ for fixed $v\in[0,1]$ of the form
\begin{equation}
	\big[\sqrt{nh}\big(\hat G_{n,h}(x,v) - G(x,v)\big)\big]_{x\in \R} \dto \G(x)_{x\in\R}\label{example_distribution_emp_eq2}
\end{equation}
where $(\G(x))_{x\in\R}$ is a centered Gaussian process and $G(x,v) = \IP(\tilde X_0(v) \le x)$ denotes the distribution function of $\tilde X_0(u)$.

Clearly, the additional localization via kernels changes the convergence rate of the empirical process. To discuss  \reff{example_distribution_emp_eq1} with the general form \reff{definition_empirical_process}, we therefore suppose that any $f\in \sF$ has a representation
\begin{equation}
    f(z,u) = D_{f,n}(u)\cdot \bar f(z,u), \quad\quad z\in(\R^d)^{\N_0}, u\in[0,1],\label{f_form_decomposition}
\end{equation}
where $\bar f$ is independent of $n$ and $D_{f,n}(u)$ is independent of $z$. For the specific example given in \reff{example_distribution_emp_eq2}, we would consider
\[
	\sF = \Big\{(z,u) \mapsto f_{x}(z,u) := \frac{1}{\sqrt{h}}K(\frac{u-v}{h})\cdot \Ii_{\{z_0 \le x\}}: x\in\R\Big\},
\]
and thus $D_{f_x,n}(u) = \frac{1}{\sqrt{h}}K(\frac{u-v}{h})$ and $\bar f_x(z,u) = \Ii_{\{z_0 \le x\}}$.

We now introduce the necessary assumptions for our empirical process theory based on the functional dependence measure. Based on the decomposition \reff{f_form_decomposition}, we define the following two function classes based on $\bar f$, which mimic the one-step-ahead mean and variance forecast,
\begin{eqnarray*}
	\bar \sF^{(1)} &:=& \{(z,u) \mapsto \IE[ \bar f(Z_i,u)|Z_{i-1}=z]: f\in \sF\},\\
	\bar \sF^{(2)} &:=& \{(z,u) \mapsto \IE[ \bar f(Z_i,u)^2|Z_{i-1}=z]^{1/2}: f\in \sF\}.
\end{eqnarray*}
For $s \in (0,1]$, a sequence $z = (z_j)_{j\in\N_0}$ of elements of $\R^d$ (equipped with the maximum norm $|\cdot|_{\infty}$) and an absolutely summable sequence $L = (L_j)_{j\in\N_0}$ of nonnegative real numbers, we set
\[
    |z|_{L,s} := \Big(\sum_{j=0}^{\infty}L_j |z_j|_{\infty}^s\Big)^{1/s}, \quad\quad |z|_{L} := |z|_{L,1}.
\]

\begin{defi} A class $\sG$ is called a $(L,s,R,C)$-class if $L = (L_{j})_{j\in\N_0}$ is a sequence of nonnegative real numbers, $s \in (0,1]$ and $R:(\R^d)^{\N_0} \times [0,1] \to [0,\infty)$ satisfies for all $u\in[0,1]$, $z,z' \in (\R^d)^{\N_0}$, $g \in \sG$,
    \[
         |g(z,u) - g(z',u)| \le |z-z'|_{L,s}^s\cdot \big[R(z,u) + R(z',u)\big].
    \]
    Furthermore, the tuple $C = (C_R,C_{\sG}) \in (0,\infty)^2$ satisfies $\sup_u|g(0,u)| \le C_{\sG}$, $\sup_u|R(0,u)| \le C_R$.
\end{defi}

There are two basic assumptions on $\bar f$ connected to our main result. The first is a compatibility condition which connects smoothness properties of $\bar\sF^{(\kappa)}$, $\kappa \in \{1,2\}$ with corresponding moment assumptions on the process $X_i$, $i=1,...,n$.

\begin{assum}[Compatibility condition on $\sF$]\label{ass2} The classes $\bar \sF^{(\kappa)}$, $\kappa \in \{1,2\}$, are $(L,s,R,C)$-classes, and there exists $p \in (1,\infty]$,  $C_X > 0$ such that
    \[
        \sup_{i,u}\|R(Z_{i-1},u)\|_{2 p} \le C_R, \quad\quad \sup_{i,j}\|X_{ij}\|_{\frac{2 sp}{p-1}} \le C_X.
    \]
    Let $\ID_n \ge 0$ and $\Delta(k) \ge 0$ such that
    \[
        2d C_R\sum_{j=0}^{k-1}L_{j}(\delta_{\frac{2 sp}{p-1}}^{X}(k-j-1))^s \le \Delta(k), \quad\quad \sup_{f\in \sF} \Big(\frac{1}{n}\sum_{i=1}^{n}\big|D_{f,n}(\frac{i}{n})\big|^{2}\Big)^{1/2} \le \ID_{n}.
\]
\end{assum}



Based on Assumption \ref{ass2}, we define for $f\in\sF$,
\begin{equation}
    V_n(f) := \|f\|_{2,n} + \sum_{k=1}^{\infty}\min\{\|f\|_{2,n},\ID_n \Delta(k)\}.\label{definition_v}
\end{equation}
Clearly, $V_n(f-g)$ is a distance between $f,g\in \sF$. Furthermore, let
\begin{equation}
	\beta(q) := \sum_{j=q}^{\infty}\Delta(j).\label{definition_betaq}
\end{equation}
Postponing some technicalities, we state our main result. In the space
\begin{equation}
    \ell^{\infty}(\sF) = \{\G:\sF \to \R \,|\, \|\G\|_{\infty} := \sup_{f\in\sF}|\G(f)| < \infty\},\label{definition_linf}
\end{equation} the following theorem holds true.

Note that it is a result of the convergence of the finite-dimensional distributions in Section \ref{sec_clt}, Theorem \ref{theorem_clt_mult_martingale_paper}, and asymptotic tightness in Section \ref{sec_asymptotic_tightness}, Corollary \ref{cor_martingale_equicont}.

\begin{thm}\label{cor_functional_central_limit_theorem_martingale}
    Suppose that $\sF$ satisfies Assumptions \ref{ass2}, \ref{ass_clt_expansion_ass2}, \ref{ass3}, \ref{ass_clt_process}, \ref{ass_clt_fcont}, \ref{ass_clt_dexplicit}. For
     \begin{equation}
        \psi(\varepsilon) = \sqrt{\log(\varepsilon^{-1} \vee 1)}\log\log(\varepsilon^{-1} \vee e)\label{definition_psi_factor}
    \end{equation}
    suppose that
    \[
        \sup_{n\in\N}\int_0^{1}\psi(\varepsilon)\sqrt{\IH(\varepsilon,\sF,V_n)} d \varepsilon < \infty.
    \]
    Then in $\ell^{\infty}(\sF)$,
    \[
        \big[\G_n(f)\big]_{f\in\sF} \dto \big[\G(f)\big]_{f\in\sF}
    \]
    where $(\G(f))_{f\in\sF}$ is a centered Gaussian process with covariances
    \[
        \Cov(\G(f),\G(g)) = \lim_{n \to \infty} \Cov(\G_n(f),\G_n(g)) = \Sigma^{(\IK)}
    \]
    and $\Sigma^{(\IK)}$ is from Assumption \ref{ass_clt_dexplicit}.
\end{thm}

Note that properties of the space $\ell^{\infty}(\sF)$ can be found in \cite{Vaart98}, for instance.

Suppose that $\ID_n \in (0,\infty)$ is independent of $n\in\N$. Based on decay rates of $\Delta(k)$, we derive simpler forms of $V_n$ which are shown below in \ref{table_v_values}. These results are proven in \cite[Lemma 7.11 and Lemma 7.12]{empproc}.

\renewcommand{\arraystretch}{2}

\begin{table}[h!]
    \centering
        \begin{tabular}{l|l|l}
         & \multicolumn{2}{c}{$\Delta(j)$ }  \\
         &  $c j^{-\alpha}$, $\alpha > 1, c > 0$         & $c \rho^j$, $\rho \in (0,1)$, $c > 0$    \\
         \hline
         \hline
        $V_n(f)$ & $\|f\|_{2,n}\max\{\|f\|_{2,n}^{-\frac{1}{\alpha}},1\}$ & $\|f\|_{2,n}\max\{\log(\norm{f}_{2,n}^{-1}),1\}$ \\ \hline
        $\int_0^\sigma \sqrt{\IH(\eps,\sF,V_n)}d\eps$ & $\int_0^{\tilde \sigma} \eps^{-\frac{1}{\alpha}}\psi(\eps) \sqrt{\IH(\eps,\sF, \|\cdot\|_{2,n})}d\eps$ & $\int_0^{\tilde\sigma} \log(\eps^{-1})\psi(\eps) \sqrt{\IH(\eps,\sF, \|\cdot\|_{2,n})}d\eps$
        \end{tabular}
        
        \caption{Equivalent expressions of $V_n$ and the corresponding entropy integral under the condition that $\ID_n \in (0,\infty)$ is independent of $n$. We omitted the lower and upper bound constants which are only depending on $c, \rho, \alpha$ and $\ID_n$. Furthermore, $\tilde \sigma = \tilde \sigma (\sigma)$ fulfills $\tilde \sigma \to 0$ for $\sigma \to 0$. }
        \label{table_v_values}
    \end{table}

The theorem significantly simplifies if $X_i$ is stationary, $\bar f(z,u) = \bar f(z_0)$, depends only on one observation and no weighting is present, i.e. $D_{f,n}(u) = 1$. Assumptions \ref{ass_clt_expansion_ass2}, \ref{ass_clt_fcont}, \ref{ass_clt_process} and \ref{ass_clt_dexplicit} are then directly fulfilled. These assumptions are needed only to provide a (pointwise) central limit theorem for locally stationary processes. They basically ask for several smoothness properties of $\bar f$. 

In More detail, let
\[
    \tilde\G_n(h) := \frac{1}{\sqrt{n}}\sum_{i=1}^{n}\big\{h(X_i) - \IE h(X_i)\big\},
\]
where $X_i = J(\sA_i)$, $i=1,...,n$, is a stationary Bernoulli shift process and $\sH \subset \{h:\R^d \to \R \text{ measurable}\}$ with envelope function $\bar h$, i.e. for $h \in \sH$ we have $|h(\cdot)| \le \bar h (\cdot)$, such that
\begin{eqnarray*}
	h^{(1)}(z_0) = \IE[h(X_1)|X_0=z_0], \quad\quad h^{(2)}(z_0) = \IE[h(X_1)^2|X_0=z_0]^{1/2}
\end{eqnarray*}
are H\"older continuous with exponent $s$ and constant $L_{\sH}$, that is, for all $z,z' \in \R$,
\[
	|h^{(1)}(z) - h^{(1)}(z')| \le L_{\sH}|z-z'|^s, \quad\quad |h^{(2)}(z) - h^{(2)}(z')| \le L_{\sH}|z-z'|^s.
\]

Assumption \ref{ass2} automatically holds with $R(\cdot) = \frac{1}{2}$ and thus $C_R = \frac{1}{2}$, $L =  L_{\sH}$ as well as $C_\sG = \max\{h^{(1)}(0), h^{(2)}(0)\}$. Then we have the following corollary of Theorem \ref{cor_functional_central_limit_theorem_martingale}.

\begin{cor}\label{corollary_empirical}
	Suppose that $\|X_1\|_{2s} < \infty$ and put $\ID_n := 1$. Let $\Delta(k)$ fulfill $\Delta(k) \ge d L_{\sH} \delta_{2s}^{X}(k-1)^s$ and there exists $C_{\beta} > 0$ such that for all $q_1,q_2\in\N$,
	\begin{equation}
		\beta(q_1 q_2) \le C_{\beta}\beta(q_1)\beta(q_2).\label{corollary_empirical_eq1}
	\end{equation}
	Furthermore, $\|\bar h(X_1)\|_{2\bar p} < \infty$ for some $\bar p > 1$.
	Assume that
	\[
		\sup_{n\in\N}\int_0^{1} \psi(\varepsilon)\sqrt{\IH(\varepsilon,\sH,V_n)} d \varepsilon < \infty,
	\]
	where $\psi(\varepsilon)$ is from \reff{definition_psi_factor}. Then it holds in $\ell^{\infty}(\sH)$ that
	\[
		\big[\tilde\G_n(h)\big]_{h\in \sH} \dto \big[\tilde \G(h)\big]_{h\in \sH},
	\]
	where $(\tilde \G(h))_{h\in \sH}$ is a centered Gaussian process with covariances
    \[
        \Cov(\tilde\G(h_1),\tilde\G(h_2)) = \sum_{k\in\Z}\Cov(h_1(X_0), h_2(X_k)).
    \]
\end{cor}

\subsection{Application to empirical distribution functions of stationary processes}

As an example, consider the family of indicators
\[
	\sH = \{h_x(z_0) := \Ii_{\{z_0 \le x\}}: x\in \R\},
\]
which is the function class corresponding to the empirical distribution function
\[
	\big[\hat G_{n}(x)\big]_{x\in\R} = \Big[\frac{1}{n}\sum_{i=1}^{n}\Ii_{\{X_i  \le x\}}\Big]_{x\in\R} = \big[\tilde \G_n(h)\big]_{h \in \sH}.
\]
Suppose that $X_i$, $i=1,...,n$, is stationary. Define the conditional distribution function
\[
	G_z(x) = \IP(X_1 \le x|X_0 = z).
\]
Then we have the following corollary.
\begin{cor}\label{corollary_empdistr}
	Suppose that $X_i$ is stationary and $z \mapsto G_z(x)$ is Lipschitz continuous with Lipschitz constant $L_{G}$ for all $x\in\R$. Suppose that for some $s \in (0,\frac{1}{2}]$, $\|X_1\|_{2s} < \infty$ and $\delta_{2s}^{X}(k) \le c k^{-\alpha}$ with $\alpha > \frac{1}{s}$, $c > 0$. Then, 
\[
	\big[\hat G_{n}(x)\big]_{x\in\R} \dto \big[\tilde\G(x)\big]_{x\in\R},
\]
where $\tilde\G(x)$ is a Gaussian process with 
\[
	\Cov(\tilde\G(x),\tilde\G(y)) = \sum_{k\in\Z}\Cov(\Ii_{\{X_0 \le x\}}, \Ii_{\{X_k \le y\}}).
\]
\end{cor}
\begin{proof}[Proof of Corollary  \ref{corollary_empdistr}]
	Due to $\min\{1,w\} \le w^a$ for $a\in [0,1]$, $w \ge 0$, we have that for any $s \in (0,\frac{1}{2}]$,
\[
	|G_z(x) - G_{z'}(x)| \le \min\{1, L_G |z-z'|\} \le L_G^{s} |z-z'|^{s}.
\]
and
\[
	|G_z(x) - G_{z'}(x)|^{1/2} \le \min\{1, (L_G |z-z'|)^{1/2}\} \le L_G^{s} |z-z'|^{s}.
\]
Choose $\Delta(k) = c L_{G} (k-1)^{-\alpha s}$, which is easily seen to satisfy \reff{corollary_empirical_eq1} (in particular, $\beta(q) < \infty$ for $q\in \N$) for some $C_{\beta} = C_{\beta}(\alpha, s, c, L_{G})$ chosen large enough.

Note that $\IH(\varepsilon,\sH,\|\cdot\|_{2,n}) = O(\log(\varepsilon^{-1}))$ for a given $\varepsilon > 0$ by \cite[Example 19.6]{Vaart98}, since in the stationary situation of the corollary, $\|h\|_{2,n} = \IE[h(X_1)^2]^{1/2}$. Since $\alpha s > 1$, Table \ref{table_v_values} implies that
\[
	\int_0^{1} \psi(\varepsilon)\sqrt{\IH(\gamma,\sH,V_n)} d \varepsilon = O\Big(\int_0^{1}\psi(\varepsilon) \varepsilon^{-\frac{1}{\alpha s}}\sqrt{\log(\varepsilon^{-1})} d \varepsilon\Big) < \infty.
\]
Corollary \ref{corollary_empirical} now implies the assertion. 
\end{proof}

\subsection{Comparison with other functional convergence results for the  empirical distribution function of stationary processes}

In the literature, several functional convergence results for the empirical distribution function were already provided. Here we list some approaches which are closely related to the functional dependence measure and compare the results to Corollary \ref{corollary_empdistr}.

In \cite{dehling2001}, stationary processes of the form $X_i = J(\sG_i)$ are considered where $\sG_i = (\varepsilon_i,\varepsilon_{i-1},...)$ and $J$ is measurable. Therein, the function $J$ itself is assumed to fulfill a (geometrically decaying) Lipschitz condition, i.e. for any sequences $(a_i),(a_i')$ with $a_i=a_i'$, $i \le k$,
\begin{equation}
    \big| J((a_i)) - J((a_i')) \big| \le C \alpha^k\label{comparison1_eq1}
\end{equation}
for some constants $C,\alpha > 0$. Based on this, 1-approximation coefficients $a_k$ are defined as upper bounds on
\[
    \IE\big\| X_0 - \IE[X_0|\sigma(\varepsilon_{0},...,\varepsilon_k)]\big\|_1 \le a_k.
\]
There is a strong connection between $\delta_{1}^{X}(k)$ and $a_k$, since it is possible to choose $a_k \le \sum_{j=k+1}^{\infty}\delta_{1}^X(j)$. The work of \cite[Theorem 5]{dehling2001} shows that under summability conditions on $a_k$, the $\beta$-mixing coefficients \emph{and} monotonicity assumptions on $\sF = \{f_t:t\in [0,1]\}$, a uniform central limit theorem for $(\G_n(f_t))_{t\in[0,1]}$ holds. Compared to our setting, \reff{comparison1_eq1} would lead to a geometrically decaying functional dependence measure $\delta^{X}$. Thus, the result in our Corollary \ref{corollary_empdistr} is much less restrictive regarding the dependency of the underlying process.

In \cite[Theorem 2.1]{dedecker10}, a uniform central limit theorem for the empirical distribution function is shown under $\beta_2(k) = O(k^{-1-\gamma})$, $\gamma > 0$, by using specifically designed dependence coefficients $\beta_2(k)$, $k\in\N_0$, based on the idea of absolute regularity. We now compare this result with Corollary \ref{corollary_empdistr}. In \cite[Section 6.1]{dedecker07} it was shown that if $X_i = J(\sG_i)$ is stationary and the distribution function of $X_1$ is Lipschitz continuous, then for any $\nu \in [0,1]$ one has
\[
    \beta_2(k) \le C\cdot \Big(\sum_{j=k+1}^{\infty}\delta_{\nu}^{X}(j)^{\nu'}\Big)^{\frac{\nu}{\nu'(\nu+1)}},\quad\quad \nu' = \min\{\nu,1\},
\]
where $C > 0$ is a constant independent of $k$. The condition $\beta_2(k) = O(k^{-1-\gamma})$ now naturally provides a decay condition on $\delta_{\nu}^{X}(k)$. With $\nu = 2s$ which corresponds to the moments of the process we have given in Corollary \ref{corollary_empdistr}, we see after a short calculation that $\beta_2(k) = O(k^{-1-\gamma})$ asks for
\[
    \alpha \ge \frac{1}{s} + \frac{\gamma}{2s} + \gamma + 1.
\]
In other words, if the results from \cite{dedecker10}, \cite{dedecker07} are transferred to the functional dependence measure setting, they need a more restrictive decay condition.


Meanwhile, \cite{berkes09} investigates strong approximations of the multivariate empirical distribution function process (that is, contrary to our approach, the results are limited to empirical distribution functions). They assume that the stationary process $X_i = J(\sG_i)$ allows for approximations $(X_{i}^{(m)})$ such that for all $m,i$,
\begin{equation}
    \IP(|X_i - X_i^{(m)}| \ge m^{-A}) \le m^{-A}\label{otherapproach_eq1}
\end{equation}
with some $A > 4$, and for any disjoint intervals $I_1,...,I_r$ of integers and any positive integers $m_1,...,m_r$, the vectors $\{X_{i}^{(m_1)}:i \in I_1\},...,\{X_i^{(m_r)}:j\in I_r\}$ are independent provided the separation between $I_k$ and $I_l$ is greater than $m_k+m_l$. Under these assumptions, \cite[Theorem 1, Corollary 1]{berkes09} shows that the empirical distribution function of $X_i$ weakly converges to some Gaussian process.

When having knowledge about the functional dependence measure, $X_i^{(m)}$ could be chosen as $X_i^{(m)} = \IE[X_i|\varepsilon_i,...,\varepsilon_{i-m}]$. Then by Markov's inequality,
\[
    \IP\big(|X_i - X_i^{(m)}| \ge m^{-A}\big) \le \frac{\|X_i - X_i^{(m)}\|_{2s}^{2s}}{m^{-2sA}} \le  \big(m^{A}\cdot \sum_{j=m+1}^{\infty}\delta_{2s}^{X}(j)\big)^{2s},
\]
so that \reff{otherapproach_eq1} leads to a decay condition on $\delta_\nu^{X}(j)$. After a short calculation, we see that \reff{otherapproach_eq1} is fulfilled if  
\[
    \alpha \ge \big(\frac{1}{2s} + 1) A + 1,
\]
again a more restrictive decay condition than given in Corollary \ref{corollary_empdistr}.


The work of \cite{durieu14} discusses the functional convergence of the multivariate empirical distribution function under a general growth condition imposed on the moments of $\sum_{i=1}^{n}\{h(X_i) - \IE h(X_i)\}$, where $h \in \sH_{\gamma}$ are H\"older continuous functions with exponent $\gamma\in (0,1]$ approximating the indicator functions. They also relate their result to the functional dependence measure.




\subsection{Application to empirical distribution functions of locally stationary processes}

In this section, we apply our theory to the  localized empirical distribution function  $\hat G_{n,h}(x,v)$ from \reff{example_distribution_emp_eq1} on a locally stationary process as motivated in the beginning of Section \ref{sec_model}. Afterwards, we compare our result with \cite{mayer2019}.

Suppose that $X_i$ is locally stationary in the sense that for each $u \in [0,1]$, there exists a stationary process $\tilde X_i(u) = J(\sA_i,u)$, $i\in\Z$, for a measurable function $J$ such that
\[
    \|X_i - \tilde X_i(\frac{i}{n})\|_{2s} \le C_Xn^{-\varsigma}, \quad\quad \|\tilde X_i(u) - \tilde X_i(u')\|_{2s} \le C_X|u-u'|^\varsigma
\] for a constant $C_X > 0$, $\varsigma \in (0,1]$, $u,u' \in [0,1]$ and $i \in \{1,...,n\}$.

Recall $G(x,v) = \IP(\tilde X_1(v) \le x)$. Define the conditional distribution function of the stationary approximation of $X_i$,
\[
	G_z(x,v) = \IP(\tilde X_1(v) \le x \mid \tilde X_0(v) = z).
\]
Finally, we have to impose a regularity assumption on the distribution function $G_i(x) := \IP(X_i \le x)$ of the locally stationary process itself. 

We have the following generalization of Corollary \ref{corollary_empdistr}. 

\begin{cor}\label{corollary_empdistr22}
	Let $v\in (0,1)$. Suppose that there exists some $L_G > 0$ such that
	\begin{itemize}
	    \item $z \mapsto G_z(x,v)$ is Lipschitz continuous with constant $L_{G}$ for all $x\in\R$,
	    \item $x \mapsto G(x,v)$ is Lipschitz continuous with constant $L_{G}$,
	    \item $x \mapsto G_i(x)$ is Lipschitz continuous with constant $L_{G}$ and \\ $\lim_{x\to -\infty}\sup_{i,n}G_i(x) = 0$, $\lim_{x\to +\infty}\inf_{i,n}G_i(x) = 1$.
	\end{itemize}
	Assume that $K:\R \to \R$ is a Lipschitz continuous kernel function with $\int K(x) dx = 1$ and support $\subset [-\frac{1}{2},\frac{1}{2}]$.
	
	Furthermore, for some $s \in (0,\frac{1}{2}]$ let $\sup_{i,n}\|X_i\|_{2s} < \infty$ and $\delta_{2s}^{X}(k) \le c k^{-\alpha}$ with $\alpha > \frac{1}{s}$, $c > 0$.
	
	Then for $h n \to \infty$, $h \to 0$, 
\[
	\big[\hat G_{n,h}(x,v)\big]_{x\in\R} \dto \big[\tilde\G(x,v)\big]_{x\in\R},
\]
where $\tilde\G(x,v)$ is a Gaussian process with 
\[
	\Cov(\tilde\G(x,v),\tilde\G(y,v)) = \int K(u)^2 du\cdot \sum_{k\in\Z}\Cov(\Ii_{\{\tilde X_0(v) \le x\}}, \Ii_{\{\tilde X_k(v) \le y\}}).
\]
\end{cor}
\begin{proof}[Proof of Corollary  \ref{corollary_empdistr22}]
	We verify the conditions of Theorem \ref{cor_functional_central_limit_theorem_martingale}. 
	By $\min\{1,w\} \le w^a$ for $a\in [0,1]$, $w \ge 0$, we have for any $s \in (0,\frac{1}{2}]$,
\[
	|G_z(x,v) - G_{z'}(x,v)| \le \min\{1, L_G |z-z'|\} \le L_G^{s} |z-z'|^{s}.
\]
and
\[
	|G_z(x,v) - G_{z'}(x,v)|^{1/2} \le \min\{1, (L_G |z-z'|)^{1/2}\} \le L_G^{s} |z-z'|^{s}.
\]
This shows Assumption \ref{ass2} with $p = \infty$, $R(\cdot) = \frac{1}{2} = C_R$.

Choose $\Delta(k) = c L_{G} (k-1)^{-\alpha s}$, which can easily be seen to satisfy Assumption \ref{ass3} (in particular, $\beta(q) < \infty$ for $q\in \N$) for some $C_{\beta} = C_{\beta}(\alpha, s, c, L_{G})$ chosen large enough.
Regarding Assumption \ref{ass_clt_expansion_ass2} we first have
\begin{eqnarray*}
    \frac{1}{c^s}\IE \sup_{L_G|a|\le c}[|\Ii_{\{\tilde Z_0(v) \le x\}}- \Ii_{\{\tilde Z_0(v)+a \le x\}}|^2] &\le& \frac{1}{c^s} \IE|\Ii_{\{\tilde Z_0(v) \le x\}} - \Ii_{\{\tilde Z_0(v) \le x-\frac{c}{L_G}\}}| \\
    &\le& \frac{1}{c^s}(\IP(\tilde Z_0(v) \le x) - \IP(\tilde Z_0(v) \le x-\frac{c}{L_G})) \\
    &\le& \frac{1}{c^s}(G_z(x,v) - G_z(x-\frac{c}{L_G},v))\\
    &\le& \frac{1}{c^s}\min\{1,c\}\le 1.
\end{eqnarray*} The envelope function is the constant $1$-function and satisfies the required condition trivially. Therefore, Assumption \ref{ass_clt_expansion_ass2} holds true.
Assumption \ref{ass_clt_fcont} is automatically satisfied for fixed $v \in (0,1)$. For Assumption \ref{ass_clt_dexplicit}, note that $D_{f,n}(u) = \frac{1}{\sqrt{h}}K(\frac{u-v}{h})$ satisfies
\[
    \frac{1}{n}\sum_{i=1}^{n}D_{f,n}(\frac{i}{n})^2 \le \frac{1}{nh}\sum_{i=1}^{n}K(\frac{i/n-v}{h})^2 \le |K|_{\infty}^2 < \infty,
\]
and $ D_{f,n}^{\infty} \le \frac{1}{\sqrt{h}}|K|_{\infty}$. Thus $\frac{D_{f,n}^{\infty}}{\sqrt{n}} \le \frac{|K|_{\infty}}{\sqrt{nh}} \to 0$, and the support satisfies $\text{supp}[D_{f,n}(\cdot)] \subset [v-h,v+h]$. Finally, $h^{1/2}D_{f,n}^{\infty} \le |K|_{\infty} < \infty$ and, since $v\in (0,1)$,
\[
    \lim_{n\to\infty}\int_0^{1} D_{f,n}(u) D_{g,n}(u) du = \lim_{n\to\infty}\frac{1}{h}\int_0^{1}K(\frac{u-v}{h})^2 du = \int K(u)^2 du.
\]
This shows all conditions of Assumption \ref{ass_clt_dexplicit} (ii).

It holds that $\IH(\varepsilon,\sH,\|\cdot\|_{2,n}) = O(\log(\varepsilon^{-1}))$ which is proven subsequently. 

Let $\varepsilon > 0$. Since $\lim_{x\to -\infty}\sup_{i,n}G_i(x) = 0$ and $\lim_{x\to +\infty}\inf_{i,n}G_i(x) = 1$, there exists $x_N = x_N(\varepsilon) > x_1 = x_1(\varepsilon) > 0$ such that $\sup_{i,n}G_i(x_1) \le \varepsilon$, $\inf_{i,n}G_i(x_1) \ge 1 - \varepsilon$. Define $x_{j+1} := x_1 + j\cdot \frac{\varepsilon^2}{L_G}$, $j = 1,2,...,N-1$ with $N = 1 +  \lceil \frac{(x_N - x_1)L_G}{\varepsilon^2}\rceil$. Put $x_0 = -\infty$ and $x_{N+1} = \infty$. Then for $j = 1,2,...,N-1$ we have
\begin{eqnarray*}
    &&\|\Ii_{\{\cdot \le x_{j+1}\}} - \Ii_{\{\cdot \le x_{j+1}\}}\|_{2,n}^2\\
    &\le& \sup_{i=1,...,n}\IE[(\Ii_{\{X_i \le x_{j+1}\}} - \Ii_{\{X_i \le x_{j}\}})^2] = \sup_{i=1,...,n}[G_i(x_{j+1}) - G_i(x_j)]\\
    &\le& L_G|x_{j+1} - x_j| \le \varepsilon^2,
\end{eqnarray*}
which shows that $[\Ii_{\{\cdot \le x_{j}\}},\Ii_{\{\cdot \le x_{j+1}\}}]$, $j=0,...,N$ are $\varepsilon$-brackets with respect to $\|\cdot\|_{2,n}$. Hence, $\IH(\varepsilon,\sH,\|\cdot\|_{2,n}) = O(\log(\varepsilon^{-1}))$.

Since $\alpha s > 1$, Table \ref{table_v_values} implies that
\[
	\int_0^{1} \psi(\varepsilon)\sqrt{\IH(\gamma,\sH,V_n)} d \varepsilon = O\Big(\int_0^{1}\psi(\varepsilon) \varepsilon^{-\frac{1}{\alpha s}}\sqrt{\log(\varepsilon^{-1})} d \varepsilon\Big) < \infty.
\]
Theorem \ref{cor_functional_central_limit_theorem_martingale} now implies the assertion. 

\end{proof}


The recently published work \cite{mayer2019} considers functional convergence of the empirical distribution function of piece-wise locally stationary processes. They impose two rather restrictive assumptions, namely they ask the functional  dependence measure to decay geometrically. Furthermore, a  Lebesgue density of the process $X_i$ has to exist, cf. \cite[assumptions (A3) and (A5)]{mayer2019}. In the above Corollary \ref{corollary_empdistr22}, we were able to provide much weaker assumptions, in particular, we only need polynomial decay of the dependence coefficients and no density assumption is made.

\subsection{Further applications}
\label{sec_further_appli}

Our theory allows for empirical process theory of general function classes. We illustrate further applications in two short examples. 

\textbf{Example 1 (Distribution of residuals):} Consider the locally stationary time series model which is defined recursively via
    \[
        X_i = m(X_{i-1},\frac{i}{n}) + \sigma(X_{i-1},\frac{i}{n})\varepsilon_i,\quad\quad i = 1,...,n,
    \]
    where $\varepsilon_i$, $i\in\Z$, is an i.i.d. sequence of random variables and $\sigma, m:\R \times [0,1] \to \R$.

Besides estimation of $m(\cdot), \sigma(\cdot)$, it may also be of interest to derive the distribution function  $G_{\varepsilon}$ of $\varepsilon_i$. Following the approach of \cite{akritas2001}, we first have to specify estimators $\hat m$, $\hat \sigma$ for $m,\sigma$, respectively, and define empirical residuals $\hat \varepsilon_i = \frac{X_{i} - \hat m(X_{i-1},i/n)}{\hat \sigma(X_{i-1},i/n)}$. Then the convergence of $(\hat G_{\varepsilon}(x))_{x\in\R}$, 
        \[
            \hat G_{\varepsilon}(x) = \frac{1}{n}\sum_{i=1}^{n}\Ii_{\{\hat \varepsilon_i \le x\}} = \frac{1}{n}\sum_{i=1}^{n}\Ii_{\{\varepsilon_i \le x\cdot \frac{\hat \sigma(X_{i-1},i/n)}{\sigma(X_{i-1},i/n)} + \frac{\hat m(X_{i-1},i/n) - m(X_{i-1},i/n)}{\sigma(X_{i-1},i/n)}\}}
        \]
        can be discussed with empirical process theory and analytic properties of $\hat m, \hat \sigma$.
        
In the following example we make use of the maximal inequality provided in Section \ref{sec_max_ineq_tight}, Corollary \ref{lemma_hoeffding_dependent_rates}.

\textbf{Example 2 (Kernel density estimation): }
Let $K:\R\to\R$ be some bounded kernel function which is Lipschitz continuous, satisfies $\int K(u) du = 1$ and has support $\subset[-\frac{1}{2},\frac{1}{2}]$. For some bandwidth $h > 0$, put $K_h(\cdot) := \frac{1}{h}K(\frac{\cdot}{h})$. 

We consider the localized density estimate of the density $g_{\tilde X_1(v)}$ of the stationary approximation $\tilde X_1(v)$,
    \[
        \hat g_{n,h}(x,v) = \frac{1}{n}\sum_{i=1}^{n}K_{h_1}(\frac{i}{n} - v) K_{h_2}(X_i - x)
    \]
where $h_1,h_2 > 0$ are bandwidths. Suppose that:
    \begin{itemize}
        \item For some $s \le \frac{1}{2}$, $\alpha > s^{-1}$, $\delta_{2s}^{X}(j) = O(j^{-\alpha})$ and $\sup_{i,n}\|X_i\|_{2s} < \infty$.
        \item There exists $p_{K} \ge 2s, C_{K} > 0$ such that for $u$ large enough, $|K(u)| \le C_{K}|u|^{-p_{K}}$.
        \item There exist constants $C_{\infty}, L_G > 0$ such that the following holds. The conditional density $g_{X_i|X_{i-1} = z}$ of $X_i$ given $X_{i-1} = z$ satisfies $|g_{X_i|X_{i-1} = z}|_{\infty} \le C_{\infty}$ and for any $x \in \R$, $z \mapsto g_{X_i|X_{i-1} = z}(x)$ is Lipschitz continuous with constant $L_G$.
    \end{itemize}
    
    We show that if $\log(n)\big(nh_1h_2^{\frac{\alpha (s\wedge \frac{1}{2})}{\alpha (s\wedge\frac{1}{2})-1}}\big)^{-1} = O(1)$,
    \begin{equation}
        \sup_{x\in\R,v\in[0,1]}\big|\hat g_{n,h}(x,v) - \IE \hat g_{n,h}(x,v)\big| = O_p\big(\sqrt{\frac{\log(n)}{nh_1h_2}}\big).\label{example_densityestimation_eq1}
    \end{equation}
    To do so, note that
    \[
        \sqrt{nh_1h_2}\big(\hat g_{n,h}(x,v) - \IE \hat g_{n,h}(x,v)\big) = \G_n(f_{x,v}),
    \]
with
\[
        \sF = \{f_{x,v}(z,u) = \sqrt{h_1}K_{h_1}(u-v)\cdot \sqrt{h_2} K_{h_2}(z - x): x \in \R, v \in [0,1]\}.
\]
To obtain \reff{example_densityestimation_eq1}, we use Corollary \ref{lemma_hoeffding_dependent_rates}. We have for $\kappa \in \{1,2\}$,
\begin{eqnarray*}
    \mu^{(\kappa)}(z) &:=& \frac{1}{h_2} \IE [K_{h_2}(X_i - x)^\kappa \mid X_{i-1} = z]^\kappa \\
    &=& \frac{1}{\sqrt{h_2}} \Big( \int K\big( \frac{y-x}{h_2} \big) f_{X_i \mid X_{i-1} = z} (y) dy \Big)^{1/\kappa} \\
    &=& h_2^{\frac{1}{\kappa}-\frac{1}{2}} \Big( \int K(w)^\kappa f_{X_i \mid X_{i-1} = z} (x+ wh_2) dw \Big)^{1/\kappa}.
\end{eqnarray*} Hence,
\begin{eqnarray*}
    &&|\mu^{(\kappa)}(z) - \mu^{(\kappa)}(z')| \\ && \qquad \le h_2^{\frac{1}{\kappa}-\frac{1}{2}} \Big( \int |K(w)|^\kappa |f_{X_i \mid X_{i-1} = z} (x+ wh_2) - f_{X_i \mid X_{i-1} = z'} (x+ wh_2)| dw \Big)^{1/\kappa}.
\end{eqnarray*}
On the other hand, $|f_{X_i \mid X_{i-1} = z} (x+ wh_2) - f_{X_i \mid X_{i-1} = z'} (x+ wh_2)| \le \min\{L_G|z-z'|,C_\infty\}$. For $s \le \frac{1}{\kappa}$, we obtain  
\begin{eqnarray*}
    |\mu^{(\kappa)}(z) - \mu^{(\kappa)}(z')| &\le& h_2^{\frac{1}{\kappa}-\frac{1}{2}} \Big( \int |K(w)|^\kappa dw\Big)^{1/\kappa} \cdot \Big[ C_\infty \min\big\{1,\frac{L_G}{C_\infty}|z-z'|\big\} \Big]^{1/\kappa} \\
    &\le& h_2^{\frac{1}{\kappa}-\frac{1}{2}} \Big( \int |K(w)|^\kappa \Big)^{1/\kappa} C_\infty^{\frac{1}{\kappa}-s} L_G^s  |z-z'|^s.
\end{eqnarray*}
    
This shows that Assumption \ref{ass2} is satisfied with $R(\cdot) = \frac{1}{2} = C_R$ and $L_{\sF} = L_G$ and $\Delta(k) = L_G(k-1)^{-\alpha s}$. As before, it is easily seen that Assumption \ref{ass3} is satisfied.

We apply Corollary \ref{lemma_hoeffding_dependent_rates} with  $\bar F = \frac{|K|_{\infty}}{\sqrt{h_2}} =: C_{\bar F,n}$. For the grids $V_n = \{i n^{-3}:i=1,...,n^3\}$, $\sX_n = \{i n^{-3}: i \in \{-2 \lceil n^{3+\frac{1}{2s}}\rceil ,..., 2 \lceil n^{3+\frac{1}{2s}}\rceil\}\}$, we obtain
    \[
        \sqrt{nh_1h_2}\sup_{x\in\sX_n, v\in V_n}\big|\hat g_{n,h}(x,v) - \IE \hat g_{n,h}(x,v)\big| = \sup_{x\in\sX_n, v\in V_n}|\G_n(f_{x,v})| = O_p\big(\sqrt{\log(n)}\big).
    \]
    The discretization of \reff{example_densityestimation_eq1}  is rather standard and postponed to the Supplementary material \ref{suppA}, Section \ref{sec_examples_supp}.

\renewcommand{\arraystretch}{2}

\section{A general central limit theorem for locally stationary processes}
\label{sec_clt}

In this section, we provide a multivariate central limit theorem for $\G_n(f)$. To guarantee a regular behavior of the asymptotic variance, we need the following four assumptions. While Assumption \ref{ass_clt_expansion_ass2} asks for smoothness of $\bar f$ in the $L^2$-sense if $X_i$ is nonstationary, Assumption \ref{ass_clt_fcont} asks the function class $\sF$ to behave smoothly in the second argument. Assumption \ref{ass_clt_process} formulates what it means for a process to be \emph{locally stationary} (cf. \cite{dahlhaus2019}). The last Assumption \ref{ass_clt_dexplicit} mainly controls the behavior of the part $D_{f,n}(u)$ of $f\in \sF$ which does not depend on the observations.



\begin{assum}\label{ass_clt_expansion_ass2}
Let $\bar F$ be an envelope function of $\{\bar f: f \in \sF\}$, that is, $|\bar f(\cdot)| \le \bar F(\cdot)$ for all $f \in \sF$. There exists $\bar p \in (1,\infty]$ such that $\sup_{i,u}\|\bar F(Z_i,u)\|_{2\bar p} < \infty$, $\sup_{v,u}\|\bar F(\tilde Z_0(v),u)\|_{2\bar p} < \infty$. Furthermore, either
\begin{itemize}
    \item $X_i$ is stationary, or
    \item for all $c > 0$ and $f\in \sF$, 
    \begin{equation}
         \sup_{u,v\in[0,1]}\frac{1}{c^s}\IE\Big[\sup_{|a|_{L_{\sF},s} \le c}\big|\bar f(\tilde Z_0(v),u) - \bar f(\tilde Z_0(v)+a,u)\big|^2\Big] < \infty.\label{ass_clt_expansion_ass2_eq1}
    \end{equation}
    Additionally, \reff{ass_clt_expansion_ass2_eq1} also holds for $\bar F$ instead of $\bar f$.
\end{itemize}
\end{assum}

\begin{assum}\label{ass_clt_fcont}
    There exists some $\varsigma \in (0,1]$ such that for every $f\in \sF$,
    \[
        |\bar f(z,u_1) - \bar f(z,u_2)| \le |u_1 - u_2|^{\varsigma}\cdot \big(\bar R(z,u_1) + \bar R(z,u_2)\big),
    \]
    and $\sup_{u,v}\|\bar R(\tilde Z_0(v),u)\|_2 < \infty$.
\end{assum}

\begin{assum}\label{ass_clt_process}
    For each $u\in[0,1]$, there exists a process $\tilde X_i(u) = J(\sA_i,u)$, $i\in\Z$, where $J$ is a measurable function. Furthermore, there exists some $C_X > 0$, $\varsigma \in (0,1]$ such that for every $i\in \{1,...,n\}$, $u_1,u_2\in [0,1]$,
    \[
        \|X_i - \tilde X_i(\frac{i}{n})\|_{\frac{2sp}{p-1}} \le C_X n^{-\varsigma}, \quad\quad \|\tilde X_i(u_1) - \tilde X_i(u_2)\|_{\frac{2sp}{p-1}} \le C_X|u_1-u_2|^{\varsigma}.
    \]
    For $\tilde Z_i(u) = (\tilde X_{i}(u), \tilde X_{i-1}(u),...)$ it holds that $\sup_{v,u}\|R(\tilde Z_0(v),u)\|_{2p} < \infty$.
\end{assum}

For $f \in \sF$, let $D^{\infty}_{f,n} := \sup_{i=1,...,n}D_{f,n}(\frac{i}{n})$.

\begin{assum}\label{ass_clt_dexplicit}
    For all $f\in \sF$, the function $\frac{D_{f,n}(\cdot)}{D_{f,n}^{\infty}}$ has bounded variation uniformly in $n$, and
    \begin{equation}
       \sup_{n\in\N}\frac{1}{n}\sum_{i=1}^{n}D_{f,n}(\frac{i}{n})^2 < \infty, \quad\quad \frac{D_{f,n}^{\infty}}{\sqrt{n}} \to 0.\label{ass_clt_d_tech_cond1}
    \end{equation}
    One of the two following cases hold.
    \begin{enumerate}
        \item Case $\IK=1$ (global): For all $f,g\in \sF$,  $u \mapsto \IE[\IE[\bar f(\tilde Z_{j_1}(u),u)|Z_0]\cdot \IE[\bar g(\tilde Z_{j_2}(u),u)|Z_0]]$ has bounded variation for all $j_1,j_2 \in\N_0$ and the following limit exists:
        \[
            \Sigma_{fg}^{(1)} := \lim_{n\to\infty}\int_0^{1}D_{f,n}(u)D_{g,n}(u) \cdot \sum_{j\in\Z}\text{Cov}(\bar f(\tilde Z_0(u), u), \bar g(\tilde Z_j(u),u)) du.
        \]
        \item Case $\IK=2$ (local): There exists a sequence $h_n > 0$ and $v\in [0,1]$ such that $\mathrm{supp} D_{f,n}(\cdot) \subset [v-h_n,v+h_n]$. It holds that
        \[
            h_n \to 0,\quad\quad \sup_{n\in\N}(h_n^{1/2}\cdot D_{f,n}^{\infty}) < \infty.
        \]
        The following limit exists for all $f,g\in \sF$:
        \[
            \Sigma_{fg}^{(2)} := \lim_{n\to\infty}\int_0^{1}D_{f,n}(u) D_{g,n}(u) du \cdot \sum_{j\in\Z}\text{Cov}(\bar f(\tilde Z_0(v), v), \bar g(\tilde Z_j(v),v)).
        \]
    \end{enumerate}
\end{assum}

Assumption \ref{ass_clt_dexplicit} looks rather technical. The first part including \reff{ass_clt_d_tech_cond1} guarantees the right normalization of $D_{f,n}(\cdot)$. The second part ensures the convergence of the asymptotic variances $\Var(\G_n(f))$ and covariances $\Cov(\G_n(f),\G_n(g))$ with respect to the behavior of $D_{f,n}(\cdot)$.

        
        

Note that Assumptions \ref{ass_clt_fcont}, \ref{ass_clt_process} and \ref{ass_clt_dexplicit} are needed to allow for \emph{very different} function classes $\sF$. In many special cases, however, some of these assumptions are automatically fulfilled. For example,
\begin{itemize}
    \item If $\bar f(z,u) = \bar f(z)$ does not depend on $u$, Assumption \ref{ass_clt_fcont} is fulfilled.
    \item If $X_i$ is stationary, Assumption \ref{ass_clt_process} is fulfilled.
    \item If $D_{f,n}(u) = 1$, Assumption \ref{ass_clt_dexplicit} is fulfilled.
\end{itemize}

It is possible to show the following analogue of a multivariate central limit theorem as in Theorem \cite[Theorem 3.4]{empproc}. The proof is similar to the proof given in \cite[Theorem 3.4]{empproc};  the only difference appears in \cite[Lemma 7.8]{empproc} for which we supply the proof in the Supplementary Material, Section \ref{sec_clt_supp}, Lemma \ref{lemma_theorem_clt_mult} under the different Assumptions \ref{ass_clt_expansion_ass2}, \ref{ass_clt_fcont}, \ref{ass_clt_process} and \ref{ass_clt_dexplicit}.

\begin{thm}\label{theorem_clt_mult_martingale_paper}
    Suppose that $\sF$ satisfies Assumptions \ref{ass2}, \ref{ass_clt_expansion_ass2}, \ref{ass_clt_process}, \ref{ass_clt_fcont} and \ref{ass_clt_dexplicit}. Let $m\in\N$, $f_1,...,f_m \in \sF$ and $\Sigma^{\IK} = \Sigma^{(\IK)}_{f_k, f_l})_{k,l=1,...,m}$. Then,
    \[
        \frac{1}{\sqrt{n}}\sum_{i=1}^{n}\Big\{\begin{pmatrix}
            f_1(Z_i,\frac{i}{n})\\
            \vdots\\
            f_m(Z_i,\frac{i}{n})
        \end{pmatrix}- \E \begin{pmatrix}
            f_1(Z_i,\frac{i}{n})\\
            \vdots\\
            f_m(Z_i,\frac{i}{n})
        \end{pmatrix}\Big\} \dto N(0,\Sigma^{\IK}),
    \]
    where $\Sigma^{(\IK)}$ is from Assumption \ref{ass_clt_dexplicit}.
\end{thm}

\section{Maximal inequalities and asymptotic tightness under functional dependence} \label{sec_max_ineq_tight}

We now provide an approach for empirical process theory if the class $\sF$ consists of nonsmooth functions. Our  approach is based on the decomposition
\[
    \G_n(f) = \G_n^{(1)}(f) + \G_n^{(2)}(f)
\]
into a martingale
\[
    \G_n^{(1)}(f) = \frac{1}{\sqrt{n}}\sum_{i=1}^{n}\big\{f(Z_i,\frac{i}{n}) - \IE[f(Z_i, \frac{i}{n})|Z_{i-1}]\big\}
\]
and a process
\[
    \G_n^{(2)}(f) = \frac{1}{\sqrt{n}}\sum_{i=1}^{n}\big\{\IE[f(Z_i, \frac{i}{n})|Z_{i-1}]-\IE f(Z_i, \frac{i}{n})\big\}
\]
which is smooth with respect to the arguments $Z_i$ if Assumption \ref{ass2} is fulfilled. The second part $\G_n^{(2)}$ can then be controlled in a similar way as done in \cite[Section 4]{empproc}, therefore this term is only discussed in the Supplementary Material. The term $\G_n^{(1)}$ is dealt with by using a Bernstein-type inequality for martingales. Observe that the conditional variance of $\G_n^{(1)}(f)$ is bounded from above by
\[
    R_n^2(f) := \frac{1}{n}\sum_{i=1}^{n}\E[f(Z_i,\frac{i}{n})^2|Z_{i-1}].
\]
The first step is now to bound $R_n^2(f)$ uniformly over $f\in\sF$.



\subsection{Maximal inequalities}
\label{sec_maximal}


Based on $\beta(\cdot)$ from \reff{definition_betaq}, we define
\[
    q^{*}(x) := \min\{q\in\N:\beta(q) \le q\cdot x\}.
\]
Set $D_{n}^{\infty}(u) := \sup_{f\in \sF}|D_{f,n}(u)|$. For $\nu \ge 2$, choose $\ID_{\nu, n}^{\infty}$ such that
\begin{equation}
   \Big(\frac{1}{n}\sum_{i=1}^{n}D_{n}^{\infty}(\frac{i}{n})^{\nu}\Big)^{1/\nu} \le \ID_{\nu,n}^{\infty}.\label{definition_dinf}
\end{equation}
Put $\ID_n^{\infty} = \ID_{2,n}^{\infty}$. Recall that $H = H(|\sF|) = 1 \vee \log |\sF|$ as in \reff{H_fix}. 
For $\delta > 0$, let
\[
    r(\delta) := \max\{r > 0: q^{*}(r)r \le \delta\}.
\]
The values for $q^\ast(\cdot)$ and $r$ under polynomial and exponential decaying $\Delta(\cdot)$ are given in the Table \ref{table_r_values} below.

 \begin{table}[h!]
    \centering
        \begin{tabular}{l|l|l}
         & \multicolumn{2}{c}{$\Delta(j)$ }  \\
         &  $C j^{-\alpha}$, $\alpha > 1$         & $C \rho^j$, $\rho \in (0,1)$    \\ 
         \hline
         \hline
        $q^*(x)$ & $\max\{x^{-\frac{1}{\alpha}},1\}$ & $\max\{\log(x^{-1}),1\}$ \\ \hline
        $r(\delta)$ & $\min\{\delta^{\frac{\alpha}{\alpha-1}},\delta\}$ & $\min\{\frac{\delta}{\log(\delta^{-1})},\delta\}$
        \end{tabular}
        
        \caption{Equivalent expressions of $q^{*}(\cdot)$ and $r(\cdot)$ taken from \cite[Lemma 7.10]{empproc}. We omitted the lower and upper bound constants which are only depending on $C, \rho, \alpha$.}
        \label{table_r_values}
    \end{table}

We have the following theorem.

\begin{thm}[Controlling the variance]\label{lemma_hoeffding_dependent}
	Let $\sF$ satisfy $|\sF| < \infty$ and Assumption \ref{ass2}. Then there exists some universal constant $c > 0$ such that the following holds. If $\sup_{f\in\sF}\|f\|_{\infty}\le M$ and $\sup_{f\in\sF}V_n(f)\le \sigma$, then
    \begin{equation}
    \E \max_{f\in \sF}\Big|R_n^2(f) - \E R_n^2(f)\Big| \le c \cdot \min_{q\in \{1,...,n\}} \Big[\ID_n r(\frac{\sigma}{\ID_n}) \sigma + C_{\Delta}(\ID_n^{\infty})^2 \beta(q) +\frac{q M^2 H}{n}\Big].\label{proposition_rosenthal_bound_mom1_result1}
\end{equation}
Furthermore,
\begin{equation}
    \E \max_{f\in \sF}\Big|R_n^2(f) - \E R_n^2(f)\Big| \le 2c\cdot \Big[\ID_n r(\frac{\sigma}{\ID_n}) \sigma +  q^{*}\big(\frac{M^2H}{n(\ID_n^{\infty})^2C_{\Delta}}\big)\frac{M^2H}{n}\Big].\label{proposition_rosenthal_bound_mom1_result2}
\end{equation}
\end{thm}

Theorem \ref{lemma_hoeffding_dependent} in conjunction with \cite[Theorem 4.1]{empproc} can be used to provide uniform convergence rates for $\G_n(f)$.

\begin{cor}[Uniform convergence rates]\label{lemma_hoeffding_dependent_rates}
    Suppose that $\sF$ satisfies $|\sF| < \infty$,  Assumption \ref{ass2} for some $\nu \ge 2$, and Assumption \ref{ass3}. Let $\bar F := \sup_{f\in\sF}\bar f$ and assume that for some $\nu_2 \in [2,\infty]$,
    \[
        C_{\bar F,n} := \sup_{i,u}\|\bar F(Z_i,u)\|_{\nu_2} < \infty.
    \]
    If
    \begin{equation}
        \sup_{n\in\N}\sup_{f\in\sF}V_n(f) < \infty, \quad\quad \sup_{n\in\N}\frac{\ID_{\nu_2,n}^{\infty}}{\ID_{n}^{\infty}} < \infty,\quad\quad\quad \sup_{n\in\N}\frac{C_{\bar F,n}^2 H}{n^{1-\frac{2}{\nu_2}}r(\frac{\sigma}{\ID_n^{\infty}})^2} < \infty,\label{lemma_hoeffding_dependent_rates_cond1}
    \end{equation}
    then
    \[
        \max_{f\in\sF}|\G_n(f)| = O_p(\sqrt{H}).
    \]
\end{cor}



\subsection{Asymptotic tightness}
\label{sec_asymptotic_tightness}

In this section, we extend the maximal inequality from Theorem \ref{lemma_hoeffding_dependent} to arbitrary (infinite) classes $\sF$. We need an additional submultiplicativity assumption on $\beta(\cdot)$ from \reff{definition_betaq}.

\begin{assum}\label{ass3}
    There exists a constant $C_{\beta} > 0$ such that for each $q_1,q_2 \in\N$,
    \[
        \beta(q_1 q_2) \le C_{\beta}\cdot \beta(q_1) \beta(q_2).
    \]
\end{assum}

It is easily seen that Assumption \ref{ass3} is fulfilled if $\Delta(k)$ follows a polynomial ($\Delta(k) = c k^{-\alpha}$ for $c > 0, \alpha > 1$) or exponential decay ($\Delta(k) = c \rho^{k}$ for $c > 0$, $\rho \in (0,1)$), cf. \cite[Lemma 7.9]{empproc}. It is generally not possible to show Assumption \ref{ass3} if $\Delta(k)$ contains a factor of the form $\frac{1}{\log(k)}$.

Recall $H(k) = 1 \vee \log(k)$. For $n\in\N$, $\delta > 0$, define
\begin{equation}
    m(n,\delta,k) := r(\frac{\delta}{\ID_n})\cdot \frac{\ID_n^{\infty}n^{1/2}}{H(k)^{1/2}}.\label{lemma_chaining_definition_m}
\end{equation}

Here, $m(n,\delta,k)$ represents the threshold for rare events in the chaining procedure. We have the following maximal inequality.

\begin{thm}\label{thm_martingale_equicont}
	Let $\sF$ satisfy Assumption \ref{ass2} and \ref{ass3}, and $F$ be some envelope function of $\sF$. Furthermore, let $\sigma > 0$ and suppose that $\sup_{f\in \sF} V_n(f) \le \sigma$. Let $\psi$ be defined as in \reff{definition_psi_factor}.
     Then there exists a universal constant $c > 0$ such that for each $\eta > 0$,
    \begin{eqnarray}
         &&\IP\Big(\sup_{f\in \sF}\big|\G_n^{(1)}(f)\big| > \eta\Big)\nonumber\\
         &\le& \frac{1}{\eta}\Big[c\Big(1 + \frac{\ID_n^{\infty}}{\ID_n} + \frac{\ID_n}{\ID_n^{\infty}}\Big)\cdot \int_0^{\sigma} \psi(\varepsilon) \sqrt{1 \vee \IH\big(\eps,\sF,V\big)} \, \mathrm{d}\eps + \sqrt{n}\big\|F\Ii_{\{F > \frac{1}{4}m(n,\sigma,\N(\frac{\sigma}{2},\sF,V_n))\}}\big\|_1\Big]\nonumber\\
         &&\quad\quad +c\Big(1 + q^{*}\big(C_{\Delta}^{-1}C_{\beta}^{-2}\big)\Big(\frac{\ID_n^{\infty}}{\ID_n}\Big)^2\Big)\int_0^{\sigma}\frac{1}{\varepsilon \psi(\varepsilon)^2}d\varepsilon.\label{thm_martingale_equicont_res1}
    \end{eqnarray}
\end{thm}

\begin{rem}
    Let $m > 0$. The chaining procedure found in \cite{nishiyama2000weak} for martingales uses the fact that for functions $f,g$ with $|f|\le g$ and $g(\cdot) > m$,
    \[
        |\G_n^{(1)}(f)| \le |\G_n^{(1)}(g)| + 2\sqrt{n}\cdot \frac{1}{n}\sum_{i=1}^{n}\IE[g(Z_i,\frac{i}{n})|Z_{i-1}] \le |\G_n^{(1)}(g)| + 2\sqrt{n}\frac{R_n^2(g)}{m}.
    \]
    Afterwards, bounds for the conditional variance $R_n^2(g)$ are applied. In our case, these bounds are not sharp enough. We therefore employ the inequality
    \[
        |\G_n^{(1)}(f)| \le |\G_n^{(1)}(g)| + 2|\G_n^{(2)}(g)| + 2\sqrt{n}\frac{\|g\|_{2,n}^2}{m}
    \]
    and are forced to use the ``smooth'' chaining technique applied on $\G_n^{(2)}(g)$ as in \cite[Theorem 4.4]{empproc} and on $R_n^2(g)$ from Theorem \ref{lemma_hoeffding_dependent}.
\end{rem}

We now obtain asymptotic equicontinuity of the process $\G_n(f)$ by using Theorem \ref{thm_martingale_equicont} for $\G_n^{(1)}$ and \cite[Theorem 4.4]{empproc} for $\G_n^{(2)}$. 

\begin{cor}\label{cor_martingale_equicont}
    Let $\sF$ satisfy the Assumptions \ref{ass2}, \ref{ass3}, \ref{ass_clt_process}, \ref{ass_clt_fcont} and \ref{ass_clt_expansion_ass2}. For $\psi$ from \reff{definition_psi_factor}, suppose that
    \begin{equation}
        \sup_{n\in\N}\int_0^{\infty}\psi(\varepsilon)\sqrt{1 \vee \IH(\varepsilon,\sF,V)} d\varepsilon < \infty.\label{cor_martingale_equicont_eq0}
    \end{equation}
    Furthermore, let $\ID_n,\ID_n^{\infty}\in (0,\infty)$ be independent of $n$, and
    \begin{equation}
        \sup_{i=1,...,n}\frac{D_{n}^{\infty}(\frac{i}{n})}{\sqrt{n}} \to 0.\label{cor_martingale_equicont_cond1}
    \end{equation}
    Then, the process $\G_n(f)$ is equicontinuous with respect to $V$, that is, for every $\eta > 0$,
    \[
        \lim_{\sigma \to 0}\limsup_{n \to \infty} \IP\Big( \sup_{f,g \in \sF, V(f-g) \leq \sigma} |\G_n(f) - \G_n(g)| \geq \eta \Big) = 0.
    \]
\end{cor}

\begin{rem}
    Compared to \cite[Corollary 4.5]{empproc}, the condition \reff{cor_martingale_equicont_eq0} of Corollary \ref{cor_martingale_equicont} is not optimal due to the additional $\log$-factor. The reason here is that we do not approximate the distance $R_n^2(\cdot)$ uniformly over the class $\sF$ in an external step but evaluate the needed bounds for $R_n^2(\cdot)$ during the chaining process. This is also the reason why our result does not include the i.i.d. version as a special case. However, in comparison to the results of \cite[Lemma 7.12]{empproc} we do not lose much due to this factor in the presence of polynomial dependence. Even in the case of exponential decay, the additional factor is of the same size as the factor already contributed due to dependence.
    
\end{rem}

\section{Conclusion}
\label{sec_conclusion} 

In this paper, we have developed an empirical process theory for locally stationary processes and function classes of possibly nonsmooth functions. Here, the dependence was quantified with the functional dependence measure. We have proven maximal inequalities and functional central limit theorems. An empirical process theory for locally stationary processes is a key step to derive asymptotic and nonasymptotic results for a large class of time series.

We have shown that our theory can be applied to empirical distribution functions (EDFs) and kernel density estimators, but much more structures can be discussed. Compared to earlier papers in the context of stationary processes and the EDF,  our results provide remarkable weak conditions on the dependence decay of the process. In particular, compared to \cite{mayer2019}, we could prove that functional weak convergence of the EDF holds under much simpler assumptions. 

From a technical point of view, the linear and moment-based nature of the functional dependence measure has forced us to modify several approaches from \cite{empproc}. A main issue was given by the fact that the dependence measure only transfers decay rates of continuous functions. The nonsmooth nature of the function class was dealt with a decomposition into a martingale and a conditional expectation part. 

\bibliographystyle{plain}
\bibliography{reference}

\newpage

\begin{center}
{\Large \bf Supplementary Material}\\
\end{center}
  
\section{Appendix}
\label{suppA}

This material contains some proof details for the main paper.

\subsection{Proofs of Section \ref{sec_model}}
\label{sec_model_supp}

\begin{lem}\label{depend_trans_2}
	Let Assumption \ref{ass2} hold for some $\nu \ge 2$. Then for all $u\in[0,1]$,
	\begin{eqnarray}
	    \delta^{\E[f(Z_i,u)|Z_{i-1}]}_{\nu}(k) &\le& |D_{f,n}(u)|\cdot \Delta(k),\label{depend_trans_2_eq22_part1}\\
	    \sup_{i}\Big\| \sup_{f\in\sF}\big|\IE[f(Z_i,u)|Z_{i-1}]\quad\quad\quad\quad\quad\quad &&\nonumber\\
	     - \IE[f(Z_i,u)|Z_{i-1}]^{*(i-k)}\big| \Big\|_{\nu} &\le& D_{n}^{\infty}(u)\cdot \Delta(k),\label{depend_trans_2_eq22_part2}\\
	    \sup_{i}\|f(Z_i,u)\|_{2} &\le& |D_{f,n}(u)|\cdot C_{\Delta}.\label{depend_trans_2_eq22_part3}
	\end{eqnarray}
	Furthermore, 
	\begin{eqnarray}
	    \Big\| \IE[f(Z_i,u)^2|Z_{i-1}] \quad\quad\quad\quad\quad\quad&&\nonumber\\ -\IE[f(Z_i,u)^2|Z_{i-1}]^{*(i-k)} \Big\|_{\nu/2} &\le& 2 |D_{f,n}(u)|\cdot \|f(Z_i,u)\|_{\nu}\cdot \Delta(k),\label{depend_trans_2_eq2}\\
		\Big\|\sup_{f\in \sF}\big| \IE[f(Z_i,u)^2|Z_{i-1}] \quad\quad\quad\quad&&\nonumber\\ -\IE[f(Z_i,u)^2|Z_{i-1}]^{*(i-k)} \big|\Big\|_{\nu/2} &\le& D_{n}^{\infty}(u)^2 \cdot C_{\Delta} \cdot \Delta(k),\nonumber\\
		&&\label{depend_trans_2_eq1}
	\end{eqnarray}
	where $C_{\Delta} := 2\max\{d,\tilde d\}|L_{\sF}|_1 C_X^s C_R + C_{\bar f}$.
\end{lem}
\begin{proof}[Proof of Lemma \ref{depend_trans_2}]
	Let $\bar\mu_{f,i}^{(1)}(z,u) = \IE[\bar f(Z_i,u)|Z_{i-1}=z]$ and $\bar\mu_{f,i}^{(2)}(z,u) = \IE[\bar f(Z_i,u)^2|Z_{i-1}=z]$. We have by Assumption \ref{ass2} that 
	\begin{eqnarray*}
	    && \sup_{i}\big\|\E[f(Z_i,u)|Z_{i-1}] - \E[f(Z_i,u)|Z_{i-1}]^{*(i-k)}\big\|_{\nu}\\
	    &=& |D_{f,n}(u)|\cdot \sup_{i}\big\|\bar\mu_{f,i}^{(1)}(Z_{i-1},u) - \bar\mu_{f,i}^{(1)}( Z_{i-1}^{*(i-k)},u)\big\|_{\nu}\\
	    &\le& |D_{f,n}(u)|\cdot \sup_{i} \norm{\big| Z_{i-1}-Z_{i-1}^{*(i-k)}\big|_{L_{\sF,s}}^{s}}_{\frac{p\nu}{p-1}}\norm{R(Z_{i-1},u) + R(Z_{i-1}^{*(i-k)},u)}_{p\nu} \\
    &\le& |D_{f,n}(u)|\cdot \sup_{i}\norm{ \sum_{j=0}^{\infty}L_{\sF,j}\big|X_{i-1-j}-X_{i-1-j}^{*(i-k)}\big|_{\infty}^s }_{\frac{p\nu}{p-1}}\\
    &&\quad\quad \times \Big(\norm{R(Z_{i-1},u)}_{p\nu} + \norm{R(Z_{i-1}^{*(i-k)},u)}_{p\nu}\Big) \\
    &\le& |D_{f,n}(u)|\cdot 2d C_R \sum_{j=0}^{k-1}L_{\sF,j} \delta_{\frac{p \nu s}{p-1}}(k-j-1)^{s},
	\end{eqnarray*}
	that is, the assertion \reff{depend_trans_2_eq22_part1} holds with the given $\Delta(k)$. The proof of \reff{depend_trans_2_eq22_part2} is similar.
	
	We now prove \reff{depend_trans_2_eq22_part3}. We have
	\[
	    \IE[f(Z_i,u)^2] = \IE[\IE[f(Z_i,u)^2|Z_{i-1}]] =  D_{f,n}(u)^2\IE[\bar\mu_{f,i}^{(2)}(Z_{i-1},u)^2]
	\]
	and thus $\|f(Z_i,u)\|_2 = |D_{f,n}(u)|\cdot \|\bar\mu_{f,i}^{(2)}(Z_{i-1},u)\|_2$. Since
	\[
	    |\bar\mu_{f,i}^{(2)}(y,u)| \le |\bar\mu_{f,i}^{(2)}(y,u) - \bar\mu_{f,i}^{(2)}(0,u)| + |\bar\mu_{f,i}^{(2)}(0,u)|,
	\]
	the proof now follows the same lines as in the proof of \cite[Lemma 7.3]{empproc}.
	
	We now show \reff{depend_trans_2_eq2} and \reff{depend_trans_2_eq1}. We have
	\[
	    \big|\bar\mu_{f,i}^{(2)}(z,u)^2 - \bar\mu_{f,i}^{(2)}(z',u)^2\big| = \big|\bar\mu_{f,i}^{(2)}(z,u) - \bar\mu_{f,i}^{(2)}(z',u)\big|\cdot \big[|\bar\mu_{f,i}^{(2)}(z,u)| + |\bar\mu_{f,i}^{(2)}(z',u)|\big].
	\]
	We then have by the Cauchy Schwarz inequality that  
	\begin{eqnarray}
		&& \Big\| \sup_{f\in \sF}\big| \bar\mu_{f,i}^{(2)}(Z_{i-1},u)^2 - \bar\mu_{f,i}^{(2)}(Z_{i-1}^{*(i-k)},u)^2\big|\, \Big\|_{\nu/2}\nonumber\\
		&\le& \Big\| \sup_{f\in \sF}\big| \bar\mu_{f,i}^{(2)}(Z_{i-1},u) - \bar\mu_{f,i}^{(2)}(Z_{i-1}^{*(i-k)},u)\big|\, \Big\|_{\nu} \cdot 2\Big\| \sup_{f\in \sF}\big|\bar\mu_{f,i}^{(2)}(Z_{i-1},u)\big|\, \Big\|_{\nu}.\label{depend_trans_2_eq3}
	\end{eqnarray}
	Since $\{\bar\mu_{f,i}^{(2)}:f\in \sF, i\in \{1,...,n\}\}$ forms a $(L_{\sF},s,R,C)$-class, the first factor in \reff{depend_trans_2_eq3} is bounded by $\Delta(k)$ as before. Furthermore,
	\begin{eqnarray*}
	    |\bar\mu_{f,i}^{(2)}(z,u)|&\le& |\bar\mu_{f,i}^{(2)}(z,u) - \bar\mu_{f,i}^{(2)}(0,u)| + |\bar\mu_{f,i}^{(2)}(0,u)|\\
	    &\le& |z|_{L_{\sF},s}^s(R(z,u) + R(0,u)) + |\bar\mu_{f,i}^{(2)}(0,u)|.
	\end{eqnarray*}
	Note that
	\begin{eqnarray*}
	    && \Big\||Z_{i-1}|_{L_{\sF},s}^s \cdot \big[R(Z_{i-1},u) + R(0,u)\big]\Big\|_{\nu}\\
	    &\le& \Big\| \sum_{j=0}^{\infty}L_{\sF,j}|Z_{i-1-j}|_{\infty}^s\Big\|_{\frac{p}{p-1}\nu}\cdot \Big(\|R(Z_{i-1},u)\|_{p\nu} + |R(0,u)|\Big)\\
	    &\le& d |L_{\sF}|_{1}\sup_{i,j}\|X_{ij}\|_{\frac{\nu sp}{p-1}}^s \cdot (C_R + |R(0,u)|)\\
	    &\le& 2d|L_{\sF}|_{1}C_X^sC_R.
	\end{eqnarray*}
	We now obtain \reff{depend_trans_2_eq1} from \reff{depend_trans_2_eq3} with the given $C_{\Delta}$.
	
	By the Cauchy-Schwarz inequality we have for $q \ge 2$,
	\begin{eqnarray}
	    &&\delta_{\nu/2}^{\E[f(Z_i,u)^2|Z_{i-1}]}(k)\nonumber\\
	    &=& \sup_i \Big\| \E[f(Z_i,u)^2|Z_{i-1}] - \E[f(Z_i,u)^2|Z_{i-1}]^{*(i-k)}\Big\|_{\nu/2}\nonumber\\
	    &=& |D_{f,n}(u)|\cdot \sup_i \Big\| D_{f,n}(u)\big(\bar\mu_{f,i}^{(2)}(Z_{i-1},u)^2 - \bar\mu_{f,i}^{(2)}(Z_{i-1}^{*(i-k)},u)^2\big)\Big\|_{\nu/2}\nonumber\\
	    &\le& |D_{f,n}(u)|\cdot \sup_i\Big\| \bar\mu_{f,i}^{(2)}(Z_{i-1},u) - \bar\mu_{f,i}^{(2)}(Z_{i-1}^{*(i-k)},u) \Big\|_{\nu}\nonumber\\
	    &&\quad\quad\quad\quad\times 2\Big\|D_{f,n}(u)\bar\mu_{f,i}^{(2)}(Z_{i-1},u)\Big\|_{\nu}\label{depend_trans_2_eq4}
	\end{eqnarray}
	Furthermore,
	\begin{equation}
	    \Big\|D_{f,n}(u)\bar\mu_{f,i}^{(2)}(Z_{i-1},u)\Big\|_{\nu} \le \|\E[f(Z_i,u)^2|Z_{i-1}]^{1/2}\|_{\nu} \le \|f(Z_i,u)\|_{\nu}.\label{depend_trans_2_eq5}
	\end{equation}
    Since Assumption \ref{ass2} holds for $\bar\mu_{f,i}^{(2)}$, the first factor in \reff{depend_trans_2_eq4} is bounded by $D_{f,n}(u)\Delta(k)$ as in the proof of \cite[Lemma  7.3]{empproc}. Inserting this and  \reff{depend_trans_2_eq5} into \reff{depend_trans_2_eq4}, we obtain the result \reff{depend_trans_2_eq2}.
\end{proof}

\subsection{Proofs of Section \ref{sec_maximal}}
\label{sec_maximal_supp}

\subsubsection{Proof of Theorem \ref{lemma_hoeffding_dependent}}

In this section, we consider
\[
    W_i(f) = \IE[f(Z_i,\frac{i}{n})^2|Z_{i-1}],\quad\quad S_n(f) := \sum_{i=1}^{n}\big\{W_i(f) - \IE W_i(f)\big\}.
\]
Then
\[
    R_n(f)^2 = \frac{1}{n}\sum_{i=1}^{n}W_i(f), \quad\quad R_n(f)^2 - \IE R_n(f)^2 = \frac{1}{n}S_n(f).
\]

We obtain from Lemma \ref{depend_trans_2}, \reff{depend_trans_2_eq2} and \reff{depend_trans_2_eq1} the following results with $\nu=2$.
\begin{lem}
    \label{ass_general2}
    Suppose that Assumption \ref{ass2} holds. Then for each $i=1,...,n$, $j\in\N$, $s\in\N \cup\{\infty\}$, $f\in\sF$,
    \begin{eqnarray*}
        \Big\| \sup_{f\in \sF}\big|W_i(f) - W_i(f)^{*(i-j)}\big|\,\Big\|_1 &\le&  C_{\Delta} D_{n}^{\infty}(\frac{i}{n})^2\Delta(j),\\
        \big\|W_i(f) - W_i(f)^{*(i-j)}\big\|_1 &\le& 2|D_{f,n}(\frac{i}{n})|\cdot \|f(Z_i,\frac{i}{n})\|_{2} \Delta(j),\\
        \big\|W_i(f)\|_s &\le& \|f(Z_i,\frac{i}{n})\|_{2s}^2.
    \end{eqnarray*}
\end{lem}

We approximate $W_i(f)$ by independent variables as follows (cf. also  \cite{Wu13}, \cite{wuzhang2017}). Let 
\[
    W_{i,j}(f) := \E[W_i(f)|\varepsilon_{i-j},\varepsilon_{i-j+1},...,\varepsilon_i],\quad\quad j\in\N,
\]
and
\[
    S_{n,j}(f) := \sum_{i=1}^{n}\{W_{i,j}(f) - \IE W_{i,j}(f)\}.
\]
Let $q \in\{1,...,n\}$ be arbitrary. Put $L := \lfloor\frac{\log(q)}{\log(2)}\rfloor$ and $\tau_l := 2^l$ ($l=0,...,L-1$), $\tau_L := q$. Then we have
\[
    W_i(f) = W_i(f) - W_{i,q}(f) + \sum_{l=1}^{L}(W_{i,\tau_{l}}(f)-W_{i,\tau_{l-1}}(f)) + W_{i,1}(f)
\]
(in the case $q = 1$, the sum in the middle does not appear) and thus
\[
    S_n(f) = \big[S_n(f) - S_{n,q}(f)\big] + \sum_{l=1}^{L}\big[S_{n,\tau_l}(f) - S_{n,\tau_{l-1}}(f)\big] + S_{n,1}(f).
\]
We write
\[
    S_{n,\tau_l}(f) - S_{n,\tau_{l-1}}(f) = \sum_{i=1}^{\lfloor \frac{n}{\tau_l}\rfloor+1}T_{i,l}(f),\quad\quad T_{i,l}(f) := \sum_{k=(i-1)\tau_l+1}^{(i\tau_l)\wedge n}\big[W_{k,\tau_l}(f) - W_{k,\tau_{l-1}}(f)\big].
\]
The random variables $T_{i,l}(f), T_{i',l}(f)$ are independent if $|i-i'| > 1$. This leads to the decomposition
\begin{eqnarray}
    \max_{f\in \sF}\Big|\frac{1}{n}S_n(f)\Big| &\le& \max_{f\in \sF}\frac{1}{n}\big|S_n(f) - S_{n,q}(f)\big|\nonumber\\
    && + \sum_{l=1}^{L}\Big[\max_{f\in \sF}\Big|\frac{1}{\frac{n}{\tau_l}}\underset{i \text{ even}}{\sum_{i=1}^{\lfloor \frac{n}{\tau_l}\rfloor + 1}}\frac{1}{\tau_l}T_{i,l}(f)\Big| + \max_{f\in \sF}\Big|\frac{1}{\frac{n}{\tau_l}}\underset{i \text{ odd}}{\sum_{i=1}^{\lfloor \frac{n}{\tau_l}\rfloor + 1}}\frac{1}{\tau_l}T_{i,l}(f)\Big|\Big]\nonumber\\
    && + \max_{f\in \sF}\frac{1}{n}\big|S_{n,1}(f)\big|\nonumber\\
    &=:& A_1 + A_2 + A_3.\label{rosenthal_main_decomposition}
\end{eqnarray}

The next result is a uniform bound on means of independent random variables.

\begin{lem}\label{lemma_elementarhoeffding}
    Assume that $Q_i(f)$, $i = 1,...,m$ are independent variables indexed by $f\in \sF$ which fulfill $\E Q_i(f) = 0$, $\frac{1}{m}\sum_{i=1}^{m}\|Q_i(f)\|_1 \le \sigma_Q$ and $|Q_i(f)| \le M_Q$ a.s. ($i=1,...,n$). Then there exists some universal constant $c > 0$ such that
    \begin{equation}
         \E \max_{f\in \sF}\Big|\frac{1}{m}\sum_{i=1}^{m}Q_i(f)\Big| \le c\Big(\sigma_Q + \frac{M_Q H}{m}\Big),\label{bernstein_1moment_expectationresult}
    \end{equation}
    where $H$ is defined by \reff{H_fix}.
\end{lem}
\begin{proof}[Proof of Lemma \ref{lemma_elementarhoeffding}]
	Let $Q_i = Q_i(f)$. By Bernstein's inequality, we have for each $f\in \sF$ that
	\begin{eqnarray*}
		 \IP\Big(\Big|\frac{1}{m}\sum_{i=1}^{m}Q_i\Big| \ge x\Big) &\le& 2\exp\Big(-\frac{1}{2}\frac{x^2}{\frac{1}{m^2}\sum_{i=1}^{m}\|Q_i\|_2^2 + x \frac{M_Q}{m}}\Big)\\
		 &\le& 2\exp\Big(-\frac{1}{2}\frac{x^2}{\frac{M_Q}{m}\cdot \sigma_Q + x \frac{M_Q}{m}}\Big),
	\end{eqnarray*}
	where we used in the last step that $\|Q_i\|_2^2 = \E[Q_i^2] \le M_Q \|Q_i\|_1$.
	
	With standard arguments (cf. the proof of Lemma 19.33 in \cite{Vaart98}), we conclude that there exists some universal constant $c_1 > 0$ with
	\begin{eqnarray*}
		\E \max_{f \in \sF}\Big|\frac{1}{m}\sum_{i=1}^{m}Q_i(f)\Big| &\le& c_1 \Big( \sqrt{H}(\frac{\sigma_Q M_Q}{m})^{1/2} + \frac{M_Q H}{m}\Big).
	\end{eqnarray*}
	The result follows by using $(\frac{H\sigma_Q M_Q}{m})^{1/2} \le 2\frac{M_QH}{m} + 2\sigma_Q$.
\end{proof}

We now prove Theorem \ref{lemma_hoeffding_dependent} based on Lemma \ref{ass_general2} and Lemma \ref{lemma_elementarhoeffding} and the decomposition \reff{rosenthal_main_decomposition}.

\begin{proof}[Proof of Theorem \ref{lemma_hoeffding_dependent}]
  	We first discuss $A_2$. We have
	\begin{eqnarray*}
		\sum_{l=1}^{L}\E \max_{f\in\sF}\frac{1}{\frac{n}{\tau_l}}\Big|\sum_{1\le i \le \lfloor \frac{n}{\tau_l}\rfloor+1,i \text{ odd}}\frac{1}{\tau_l}T_{i,l}(f)\Big|.
	\end{eqnarray*}
	Since $\|W_{k,j}(f)-W_{k,j-1}(f)\|_1 \le 2\min\{\|W_k(f)\|_1, \delta^{W_k(f)}_1(j-1)\}$, we have for each $f\in \sF$,
	\begin{eqnarray}
		\frac{1}{\tau_l}\|T_{i,l}\|_1 &\le& \sum_{j=\tau_{l-1}+1}^{\tau_l}\frac{1}{\tau_l}\Big\|\sum_{k=(i-1)\tau_l+1}^{(i\tau_l) \wedge n}(W_{k,j} - W_{k,j-1})\Big\|_1\nonumber\\
		&\le& \sum_{j=\tau_{l-1}+1}^{\tau_l}\frac{1}{\tau_l}\sum_{k=(i-1)\tau_l+1}^{(i\tau_l) \wedge n}\Big\|W_{k,j} - W_{k,j-1}\Big\|_1\nonumber\\
		&\le& 2 \sum_{j=\tau_{l-1}+1}^{\tau_l}\frac{1}{\tau_l}\sum_{k=(i-1)\tau_l+1}^{(i\tau_l) \wedge n}\min\{\|W_k(f)\|_1, \delta_1^{W_k(f)}(j-1)\}\nonumber\\
		&\le& 2 \sum_{j=\tau_{l-1}+1}^{\tau_l}\min\{\frac{1}{\tau_l}\sum_{k=(i-1)\tau_l+1}^{(i\tau_l) \wedge n}\|W_k(f)\|_1, \frac{1}{\tau_l}\sum_{k=(i-1)\tau_l+1}^{(i\tau_l) \wedge n} \delta_1^{W_k(f)}(j-1)\}\nonumber\\
		&=& 2\sum_{j=\tau_{l-1}+1}^{\tau_l}\min\{\sigma_{i,l}, \Delta_{i,j,l}\},\nonumber
	\end{eqnarray}
	where
	\[
	    \sigma_{i,l} := \frac{1}{\tau_l}\sum_{k=(i-1)\tau_l+1}^{(i\tau_l) \wedge n}\|W_k(f)\|_1, \quad\quad \Delta_{i,j,l} := \frac{1}{\tau_l}\sum_{k=(i-1)\tau_l+1}^{(i\tau_l) \wedge n} \delta_1^{W_k(f)}(j-1).
	\]
	We conclude that
	\begin{eqnarray}
	   \frac{1}{\lfloor \frac{n}{\tau_l}\rfloor + 1}\sum_{i=1}^{\lfloor \frac{n}{\tau_l}\rfloor + 1} \frac{1}{\tau_l}\|T_{i,l}\|_1 &\le& 2\sum_{j=\tau_{l-1}+1}^{\tau_l}\min\{\frac{1}{\frac{n}{\tau_l}}\sum_{i=1}^{\lfloor \frac{n}{\tau_l}\rfloor + 1}\sigma_{i,l}, \frac{1}{\frac{n}{\tau_l}}\sum_{i=1}^{\lfloor \frac{n}{\tau_l}\rfloor + 1}\Delta_{i,j,l}\}\nonumber\\
	   &\le& \sum_{j=\tau_{l-1}+1}^{\tau_l}\min\{\frac{1}{n}\sum_{i=1}^{n}\|W_i(f)\|_1, \frac{1}{n}\sum_{i=1}^{n}\delta_1^{W_i}(j)\}.\label{lemma_hoeffding_dependent_boundz1}
	\end{eqnarray}
	
	Furthermore, it holds that
	\begin{equation}
	    \frac{1}{\tau_l}|T_{i,l}| \le 2\sup_i\|W_i(f)\|_{\infty} \le 2 \|f\|_{\infty}^2 \le 2M^2\label{lemma_hoeffding_dependent_boundz2}.
	\end{equation}
	
	By Lemma \ref{lemma_elementarhoeffding}, \reff{bernstein_1moment_expectationresult}, we have with some universal constant $c_1 > 0$ that
	\begin{eqnarray}
		\E A_2 &\le&  2c_1\sum_{l=1}^{L}\Big[ \sup_{f\in \sF}\Big(\frac{1}{\lfloor \frac{n}{\tau_l}\rfloor + 1}\sum_{i=1}^{\lfloor \frac{n}{\tau_l}\rfloor + 1}\frac{1}{\tau_l}\|T_{i,l}(f)\|_1\Big) + \frac{2M^2 H}{\lfloor \frac{n}{\tau_l}\rfloor + 1}\Big]\nonumber\\
		&\le& 2c_1\Big( \sum_{l=1}^{L}\sup_{f\in\sF}\sum_{j=\tau_{l-1}+1}^{\tau_l}\min\{\frac{1}{n}\sum_{i=1}^{n}\|W_i(f)\|_1, \frac{1}{n}\sum_{i=1}^{n}\delta_1^{W_i(f)}(j)\} + \frac{q M^2 H}{n}\Big).\label{lemma_hoeffding_dependent_eq1}
	\end{eqnarray}
	By Lemma \ref{ass_general2} and the Cauchy-Schwarz inequality for sums,
	\begin{eqnarray}
	    && \sum_{l=1}^{L}\sup_{f\in\sF}\sum_{j=\tau_{l-1}+1}^{\tau_l}\min\{\frac{1}{n}\sum_{i=1}^{n}\|W_i(f)\|_1, \frac{1}{n}\sum_{i=1}^{n}\delta_1^{W_i}(j)\}\nonumber\\
	    &\le& \sum_{l=1}^{L}\sup_{f\in\sF}\sum_{j=\tau_{l-1}+1}^{\tau_l}\min\{\frac{1}{n}\sum_{i=1}^{n}\|f(Z_i,\frac{i}{n})\|_2^2, \frac{2}{n}\sum_{i=1}^{n}D_{f,n}(\frac{i}{n})\|f(Z_i,\frac{i}{n})\|_2\cdot \Delta(j)\}\nonumber\\
	    &\le& \sum_{j=1}^{\infty}\min\{\sup_{f\in\sF} \|f\|_{2,n}^2, 2\ID_n \sup_{f\in\sF}\|f\|_{2,n}\cdot \Delta(j)\}\nonumber\\
	    &=& \sup_{f\in\sF}\|f\|_{2,n}\cdot \bar V(\sup_{f\in\sF}\|f\|_{2,n})\nonumber\\
	    &=& \sup_{f\in\sF}\big(\|f\|_{2,n}\cdot \bar V(\|f\|_{2,n})\big)\nonumber\\
	    &\le& \sup_{f\in\sF}\big[\|f\|_{2,n} V_n(f)\big],\label{lemma_hoeffding_dependent_eq15}
	\end{eqnarray}
	where
    \begin{equation}
    \bar V(x) = x +  \sum_{j=1}^{\infty}\min\{x,\ID_n \Delta(j)\}\label{definition_vbar}
    \end{equation} and in the second-to-last equality the fact that $x \mapsto x\cdot \bar V(x)$ is increasing in $x$.

    We also have $\|W_{i,0}(f) - \E W_{i,0}(f)\|_{\infty} \le 2\|f\|_{\infty}^2 \le 2M^2$ and $\|W_{i,0}(f) - \E W_{i,0}(f)\|_1 \le 2\|W_{i}(f)\|_1$. Thus by Lemma \ref{lemma_elementarhoeffding}, \reff{bernstein_1moment_expectationresult},
	\begin{eqnarray}
	    \E A_3 &\le& \E \max_{f\in \sF}\Big|\frac{1}{n}\sum_{i=1}^{n}(W_{i,0}(f) - \E W_{i,0}(f))\Big|\nonumber\\
		&\le& 2c_1\Big(\sup_{f\in\sF}\frac{1}{n}\sum_{i=1}^{n}\|W_i(f)\|_1 + \frac{M^2 H}{n}\Big) \\ 
		&\le& 2c_1\Big(\sup_{f\in\sF}\|f\|_{2,n}^2+ \frac{M^2 H}{n}\Big).\label{lemma_hoeffding_dependent_eq2}
	\end{eqnarray}
	
	Finally,
	\begin{eqnarray*}
		\E A_1 &\le& \sum_{j=q}^{\infty}\E \sup_{f\in \sF}\Big|\frac{1}{n}\sum_{i=1}^{n}(W_{i,j+1}(f) - W_{i,j}(f))\Big|\\
		&\le& \sum_{j=q}^{\infty}\frac{1}{n}\sum_{i=1}^{n}\big\| \sup_{f\in \sF}|W_{i,j+1}(f) - W_{i,j}(f)|\big\|_1.
	\end{eqnarray*}
	Since $|W_{i,j+1}(f) - W_{i,j}(f)| = |\E[W_i(f)^{**(i-j)} - W_i(f)^{**(i-j+1)}|\sA_i]| \le \E[ |W_i(f)^{**(i-j)} - W_i(f)^{**(i-j+1)}|\, |\sA_i]$ where we use the notation $H(\sF_i)^{**(i-j)} := H(\sF_i^{**(i-j)})$ and $\sF_i^{**(i-j)} = (\varepsilon_i,\varepsilon_{i-1},...,\varepsilon_{i-j},\varepsilon_{i-j-1}^{*},\varepsilon_{i-j-2}^{*},...)$., we have
	\begin{eqnarray}
		&&\big\| \sup_{f\in \sF}|W_{i,j+1}(f) - W_{i,j}(f)|\big\|_1\nonumber\\
		&\le& \big\| \E[ \max_{f\in \sF} |W_i(f)^{**(i-j)} - W_i(f)^{**(i-j+1)}|\, |\sA_i]\big\|_1\nonumber\\
		&\le& \big\| \sup_{f\in \sF} |W_i(f)^{**(i-j)} - W_i(f)^{**(i-j+1)}| \big\|_1\nonumber\\
		&=& \big\| \sup_{f\in \sF} |W_i(f) - W_i(f)^{*(i-j)}| \big\|_1 \le D_n^{\infty}(\frac{i}{n})^2C_{\Delta}\Delta(j),\label{lemma_hoeffding_dependent_eq3_partialresult}
	\end{eqnarray}
	which shows that
	\begin{equation}
		\E A_1 \le (\ID_n^{\infty})^2C_{\Delta}\beta(q).\label{lemma_hoeffding_dependent_eq3}
	\end{equation}
	
	Collecting the upper bounds \reff{lemma_hoeffding_dependent_eq1}, \reff{lemma_hoeffding_dependent_eq15},  \reff{lemma_hoeffding_dependent_eq2} and \reff{lemma_hoeffding_dependent_eq3}, we obtain that
	\begin{equation}
	    \IE \max_{f\in\sF}\Big|\frac{1}{n}S_n(f)\Big| \le (4c_1+1)\cdot \Big[\sup_{f\in\sF}\big[\|f\|_{2,n}V_n(f)\big] + (\ID_n^{\infty})^2 C_{\Delta}\beta(q) + \frac{qM^2H}{n}\Big].\label{lemma_hoeffding_dependent_eq4}
	\end{equation}
	
	By \reff{proposition_rosenthal_bound_implication3_bound2norm}, $V_n(f) \le \sigma$ implies $\|f\|_{2,n}^2 \le \ID_n r(\frac{\delta}{\ID_n})\|f\|_{2,n}$ and thus
	\[
	    \|f\|_{2,n} \le \ID_n r(\frac{\sigma}{\ID_n}),
	\]
	thus
	\begin{equation}
	    \sup_{f\in\sF}\big[\|f\|_{2,n} V_n(f)\big] \le \ID_n r(\frac{\sigma}{\ID_n}) \sigma.\label{lemma_hoeffding_dependent_eq5}
	\end{equation}
	Inserting \reff{lemma_hoeffding_dependent_eq5} into \reff{lemma_hoeffding_dependent_eq4}  yields the first assertion \reff{proposition_rosenthal_bound_mom1_result1} of the lemma.

	We now show \reff{proposition_rosenthal_bound_mom1_result2} with a case distinction. We abbreviate $q^{*} = q^{*}(\frac{M^2H}{n(\ID_n^{\infty})^2C_{\Delta}})$. If  $q^{*}\frac{H}{n} \le 1$, we have $q^{*} \in \{1,...,n\}$ and thus
    \begin{eqnarray}
        P &\le& c\Big(\ID_n r(\frac{\sigma}{\ID_n}) \sigma + (\ID_n^{\infty})^2C_{\Delta}\beta(q^{*}) + q^{*}\frac{M^2H}{n}\Big)\nonumber\\
        &\le& 2c\Big(\ID_n r(\frac{\sigma}{\ID_n}) \sigma + q^{*}\frac{M^2H}{n}\Big)\nonumber\\
        &=& 2c\Big(\ID_n r(\frac{\sigma}{\ID_n}) \sigma + M^2\cdot \min\Big\{q^{*}\frac{H}{n},1\Big\}\Big).\label{proposition_rosenthal_bound_mom1_result2_eq1}
    \end{eqnarray}
    If $q^{*}\frac{H}{n} \ge 1$, choose $q_0 = \lfloor \frac{n}{H}\rfloor \le \frac{n}{H}$. By simply bounding each summand with $M^2$, we have
    \begin{eqnarray}
        \E \max_{f\in \sF}\Big|\frac{1}{n}S_n(f)\Big| &\le& M^2 \le c\Big(\ID_n r(\frac{\sigma}{\ID_n}) \sigma + M^2\Big)\nonumber\\
        &\le& 2c\Big(\ID_n r(\frac{\sigma}{\ID_n}) \sigma + M^2\cdot \min\Big\{q^{*}\frac{H}{n},1\Big\}\Big).\label{proposition_rosenthal_bound_mom1_result2_eq2}
    \end{eqnarray}
    holds. Putting the two bounds \reff{proposition_rosenthal_bound_mom1_result2_eq1} and \reff{proposition_rosenthal_bound_mom1_result2_eq2} together, we obtain the result \reff{proposition_rosenthal_bound_mom1_result2}.
\end{proof}

The following lemma is an auxiliary result to prove Corollary \ref{lemma_hoeffding_dependent_rates} and Lemma \ref{lemma_chaining2}.

\begin{lem}\label{lemma_maximal_inequality_martingales_part1}
	Let $\sF$ be some finite class of functions. Let $R > 0$ be arbitrary and assume that $\sup_{f\in\sF}\|f\|_{\infty} \le M$. Then  there exists a universal constant $c > 0$ such that
	\begin{equation}
		\E \max_{f\in \sF}\big|\G_n^{(1)}(f)\big|\Ii_{\{R_n(f)^2 \le R^2\}} \le c\Big\{R \sqrt{H} + \frac{M H}{\sqrt{n}}\Big\},\label{lemma_maximal_inequality_martingales_eq1}
	\end{equation}
	where $H$ is defined by \reff{H_fix}.
\end{lem}
\begin{proof}[Proof of Lemma \ref{lemma_maximal_inequality_martingales_part1}]
By Theorem 3.3 in \cite{pinelis1994}, it holds for $x,a > 0$ and a measurable function $f$ that
\[
    \IP\Big(\big|\G_n^{(1)}(f)\big| \ge x, R_n(f)^2 \le R^2\Big) \le 2\exp\Big(-\frac{1}{2}\frac{x^2}{R^2 + \frac{2\|f\|_{\infty}x}{3\sqrt{n}})}\Big).
\]
Using standard arguments (cf. the proof of Lemma 19.33 in \cite{Vaart98}), we obtain \reff{lemma_maximal_inequality_martingales_eq1}.
\end{proof}

\begin{proof}[Proof of Corollary \ref{lemma_hoeffding_dependent_rates}]
Let us define the following functions first.

For $m > 0$, define $\varphi_m^{\wedge}:\R \to \R$ and the corresponding ``peaky'' residual function $\varphi_m^{\vee}:\R \to \R$ via  
\[
    \varphi_{m}^{\wedge}(x) := (x\vee (-m))\wedge m, \quad\quad 
    \varphi_m^{\vee}(x) := x - \varphi_m^{\wedge}(x).
\]

    Now, let $Q \ge 1$, and $\sigma := \sup_{n\in\N}\sup_{f\in\sF}V_n(f) < \infty$. Put
    \[
        M_n = \frac{\sqrt{n}}{\sqrt{H}}r\big(\frac{\sigma Q^{1/2}}{\ID_n^{\infty}}\big) \ID_n^{\infty}.
    \]
    Let $F(z,u) := D_n^{\infty}(u)\cdot \bar F(z,u)$, (recall $\bar F = \sup_{f\in\sF}\bar f$). Then
    \begin{eqnarray}
        \IP\Big(\max_{f\in\sF}|\G_n(f)| > Q \sqrt{H}\Big) &\le& \IP\Big(\max_{f\in\sF}|\G_n(f)| > Q \sqrt{H}, \sup_{i=1,...,n}F(Z_i,\frac{i}{n}) \le M_n\Big)\nonumber\\
        &&\quad\quad + \IP\Big(\sup_{i=1,...,n}F(Z_i,\frac{i}{n}) > M_n\Big)\nonumber\\
        &\le& \IP\Big(\max_{f\in\sF}|\G_n(\varphi_{M_n}^{\wedge}(f))| > \frac{Q \sqrt{H}}{2}\Big)\nonumber\\
        &&\quad\quad + \IP\Big(\frac{1}{\sqrt{n}}\max_{f\in\sF}\big|\sum_{i=1}^{n}\IE[f(Z_i,\frac{i}{n})\Ii_{\{|f(Z_i,\frac{i}{n})| > M_n\}}]\big| > \frac{Q \sqrt{H}}{2}\Big)\nonumber\\
        &&\quad\quad + \IP\Big(\sup_{i=1,...,n}F(Z_i,\frac{i}{n}) > M_n\Big).\label{lemma_hoeffding_dependent_rates_eq1}
    \end{eqnarray}
    For the first summand in \reff{lemma_hoeffding_dependent_rates_eq1}, we use the decomposition
    \begin{eqnarray}
        &&\IP\Big(\max_{f\in\sF}|\G_n(\varphi_{M_n}^{\wedge}(f))| > \frac{Q \sqrt{H}}{2}\Big)\nonumber\\
        &\le& \IP\Big(\max_{f\in\sF}|\G_n^{(1)}(\varphi_{M_n}^{\wedge}(f))| > \frac{Q \sqrt{H}}{4}\Big) + \IP\Big(\max_{f\in\sF}|\G_n^{(2)}(\varphi_{M_n}^{\wedge}(f))| > \frac{Q \sqrt{H}}{4}\Big)\nonumber\\
        &\le& \IP\Big(\max_{f\in\sF}|\G_n^{(1)}(\varphi_{M_n}^{\wedge}(f))| > \frac{Q \sqrt{H}}{4}, \, \max_{f\in\sF}R_n(\varphi_{M_n}^{\wedge}(f))^2 \le \sigma^2\Big)\nonumber\\
        &&\quad\quad + \IP\Big(\max_{f\in\sF}R_n(\varphi_{M_n}^{\wedge}(f)) > \sigma^2\Big)\nonumber\\
        &&\quad\quad + \IP\Big(\max_{f\in\sF}|\G_n^{(2)}(\varphi_{M_n}^{\wedge}(f))| > \frac{Q \sqrt{H}}{4}\Big).\label{lemma_hoeffding_dependent_rates_eq2}
    \end{eqnarray}
    We now discuss the three terms separately. By Lemma \ref{lemma_maximal_inequality_martingales_part1}, we have
    \begin{eqnarray*}
        &&\IP\Big(\max_{f\in\sF}|\G_n^{(1)}(\varphi_{M_n}^{\wedge}(f))| > \frac{Q \sqrt{H}}{4}, \, \max_{f\in\sF}R_n(\varphi_{M_n}^{\wedge}(f))^2 \le Q^{3/2}\sigma^2\Big)\\
        &\le& \frac{4c}{Q\sqrt{H}}\Big[\sigma Q^{3/4}\sqrt{H} + \frac{M_n H}{\sqrt{n}}\Big] \le \frac{4c}{Q\sqrt{H}}\Big[\sigma Q^{3/4}\sqrt{H} + \sigma \sqrt{H}Q^{1/2}\Big] \le \frac{8c}{Q^{1/4}}.
    \end{eqnarray*}
    By Theorem \ref{lemma_hoeffding_dependent} and \reff{rule_submult_qstar},
    \begin{eqnarray*}
        &&\IP\Big(\max_{f\in \sF}R_n(\varphi_{M_n}^{\wedge}(f))^2 > Q^{3/2}\sigma^2\Big)\\
        &\le& \frac{2c}{\sigma^2 Q^{3/2}}\Big[\ID_n r(\frac{\sigma}{\ID_n})\sigma + q^{*}\Big(\frac{M^2 H}{n(\ID_n^{\infty})^2 C_{\Delta}}\Big)\frac{M^2 H}{n}\Big]\\
        &\le& \frac{2c}{\sigma^2 Q^{3/2}}\Big[\sigma^2 + q^{*}\Big(\frac{r(\frac{\sigma Q^{1/2}}{\ID_n^{\infty}})^2}{C_{\Delta}}\Big)r(\frac{\sigma Q^{1/2}}{\ID_n^{\infty}})^2 (\ID_n^{\infty})^2\Big]\\
        &\le& \frac{2c}{\sigma^2 Q^{3/2}}\Big[\sigma^2 + q^{*}\Big(C_{\Delta}^{-1}C_{\beta}^{-2}\big)\cdot \Big[q^{*}\Big(r(\frac{\sigma Q^{1/2}}{\ID_n^{\infty}})\Big)r(\frac{\sigma Q^{1/2}}{\ID_n^{\infty}})\Big]^2 (\ID_n^{\infty})^2\Big]\\
        &\le& \frac{2c}{\sigma^2 Q^{3/2}}\Big[\sigma^2 + q^{*}\Big(C_{\Delta}^{-1}C_{\beta}^{-2}\big)\sigma^2 Q\Big]|\\
        &\le& \frac{2c}{Q^{1/2}}\big[1 + q^{*}\Big(C_{\Delta}^{-1}C_{\beta}^{-2}\big)\big]
    \end{eqnarray*} for $C_{\Delta}$ defined in Lemma \ref{depend_trans_2}.

    By \cite[Theorem 4.1]{empproc} applied to $W_i(f) = \IE[f(Z_i,\frac{i}{n})|Z_{i-1}]$,
    \begin{eqnarray*}
        && \IP\Big(\max_{f\in\sF}|\G_n^{(2)}(\varphi_{M_n}^{\wedge}(f))| > \frac{Q \sqrt{H}}{4}\Big)\\
        &\le& \frac{8c}{Q\sqrt{H}}\cdot \Big[\sigma \sqrt{H} + q^{*}\Big(r(\frac{\sigma Q^{1/2}}{\ID_n^{\infty}})\Big)r(\frac{\sigma Q^{1/2}}{\ID_n^{\infty}}) \ID_n^{\infty}\Big]\\
        &\le& \frac{8c}{Q\sqrt{H}}\big[\sigma \sqrt{H} + \sigma Q^{1/2}\sqrt{H}\big] \le \frac{16c \sigma}{Q^{1/2}}.
    \end{eqnarray*}
    Inserting the upper bounds into \reff{lemma_hoeffding_dependent_rates_eq2}, we obtain
    \[
        \IP\Big(\max_{f\in\sF}|\G_n(\varphi_{M_n}^{\wedge}(f))| > \frac{Q \sqrt{H}}{2}\Big) \le \frac{8c}{Q^{1/4}} + \frac{2c}{Q^{1/2}}\big[1 + q^{*}\Big(C_{\Delta}^{-1}C_{\beta}^{-2}\big)\big] + \frac{16c \sigma}{Q^{1/2}} \to 0
    \]
    for $Q \to \infty$. The second and third summand in \reff{lemma_hoeffding_dependent_rates_eq1} were already discussed in the proof of \cite[Corollary 4.3]{empproc} (equation (7.34) and (7.35) therein; note especially that we only need there that $\|\bar F(Z_i,\frac{i}{n})\|_{\nu_2} \le C_{\bar F,n}$ instead of $C_{\Delta}$ which is part of the assumptions), and converge to $0$ for $Q \to \infty$ under the given assumptions.
\end{proof}

The following Lemma \ref{lemma_chaining2} is used to prove Theorem \ref{thm_martingale_equicont}.

\begin{lem}[Compatibility lemma 2]\label{lemma_chaining2}
    Let $\psi:(0,\infty)\to[1,\infty)$ be some function and $k\in\N$, $\delta > 0$. If $\sF$ fulfills $|\sF| \le k$ and Assumptions \ref{ass2}, \ref{ass3}, then there exists some universal constant $c > 0$ such that the following holds: If $\sup_{f\in \sF}V_n(f) \le \delta$ and $\sup_{f\in \sF}\|f\|_{\infty} \le m(n,\delta,k)$, then
    \begin{eqnarray}
        \E \max_{f\in \sF}\big|\G_n^{(1)}(f)\big|\Ii_{\{R_n(f) \le 2\delta \psi(\delta)\}} &\le& 2c(1+\frac{\ID_n^{\infty}}{\ID_n})\cdot \psi(\delta)\delta \sqrt{H(k)},\label{lemma_chaining2_res1}\\
        \IP\Big(\sup_{f\in\sF}R_n(f) > 2\delta\psi(\delta)\Big) &\le& \frac{2c(1+q^{*}\big(C_{\Delta}^{-1}C_{\beta}^{-2}\big)(\frac{\ID_n^{\infty}}{\ID_n})^2)}{\psi(\delta)^2}.\label{lemma_chaining2_res3}
    \end{eqnarray}
\end{lem}

\begin{proof}[Proof of Lemma \ref{lemma_chaining2}]
     By Lemma \ref{lemma_maximal_inequality_martingales_part1} and since $r(a) \le a$ (cf. \cite[Lemma 7.5]{empproc}), 
    \begin{eqnarray*}
        \E \max_{f\in \sF}\big|\G_n^{(1)}(f)\big|\Ii_{\{R_n(f) \le 2\delta \psi(\delta)\}} &\le& c\Big\{2\psi(\delta)\delta\sqrt{H(k)} + \frac{m(n,\delta,k) H(k)}{\sqrt{n}}\Big\}\\
        &\le& 2c \cdot \big[\psi(\delta)\cdot \delta + \ID_n^{\infty}r(\frac{\delta}{\ID_n})\big]\sqrt{H(k)}\\
        &\le& 2c \cdot (1 + \frac{\ID_n^{\infty}}{\ID_n})\cdot  \psi(\delta) \delta \sqrt{H(k)},
    \end{eqnarray*}
    which shows \reff{lemma_chaining2_res1}.
    
    For $\hat a = \arg\min_{j\in\N}\big\{\|f\|_{2,n} \cdot j + \ID_n\beta(j)\big\}$
    and since $\|f\|_{2,n} \le V_n(f) \le \delta$ we have with $r(\frac{\delta}{\ID_n}) \ge \frac{\delta}{\ID_n\hat a}$,
    \begin{equation}
        \frac{\|f\|_{2,n}^2}{\ID_n^{\infty} r(\frac{\delta}{\ID_n})} \le \frac{\ID_n\hat a \|f\|_{2,n}^2}{\ID_n^{\infty}\delta} \le \frac{\ID_n V_n(f) \|f\|_{2,n}}{\ID_n^{\infty}\delta} \le \frac{\ID_n}{\ID_n^{\infty}}\|f\|_{2,n}. \label{proposition_rosenthal_bound_implication3_bound2norm}
    \end{equation} Therefore, $\|f\|_{2,n}^2 \le \ID_n r(\frac{\delta}{\ID_n})\|f\|_{2,n}$ and thus $\|f\|_{2,n} \le \ID_n r(\frac{\delta}{\ID_n})$. Note that due to $r(a) \le a$,
    \begin{eqnarray}
        \IE R_n(f)^2 &=& \frac{1}{n}\sum_{i=1}^{n}\IE[f(Z_i,\frac{i}{n})^2] \le \|f\|_{2,n}^2 \le (\ID_n r(\frac{\delta}{\ID_n}))^2 \le \delta^2.\label{bound_r}
    \end{eqnarray}
    Recall that $\beta_{norm}(q) = \frac{\beta(q)}{q}$. By Assumption \ref{ass3}, we have that for any $x_1,x_2 > 0$, $\tilde q = q^{*}(x_1)q^{*}(x_2)$ satisfies
    \[
        \beta_{norm}(\tilde q) \le C_{\beta} \beta_{norm}(q^{*}(x_1)) \beta_{norm}(q^{*}(x_2)) \le C_{\beta} x_1 x_2.
    \]
    Thus, by definition of $q^{*}$,
    \begin{equation}
        q^{*}(C_{\beta}x_1 x_2) \le q^{*}(x_1) q^{*}(x_2).\label{rule_submult_qstar}
    \end{equation}
    We obtain that
    \begin{equation}
        q^{*}\Big(r(\frac{\delta}{\ID_n})^2 \frac{1}{C_{\Delta}}\Big) \le q^{*}\Big(r(\frac{\delta}{\ID_n})\Big)^2 q^{*}\big(C_{\Delta}^{-1}C_{\beta}^{-2}\big).\label{bound_qstar}
    \end{equation}
    By \reff{bound_r}, Markov's inequality,  Theorem \ref{lemma_hoeffding_dependent} and \reff{bound_qstar},
    \begin{eqnarray*}
        && \IP\Big(\sup_{f\in\sF}R_n(f)^2 > 2\psi(\delta)^2 \delta^2\Big)\\
        &\le& \IP\Big(\sup_{f\in\sF}|R_n(f)^2 - \E R_n(f)^2| > \psi(\delta)^2 \delta^2\Big)\\
        &\le& \frac{2c}{\psi(\delta)^2 \delta^2}\cdot \Big[\ID_n r(\frac{\delta}{\ID_n})\delta + q^{*}\Big(r(\frac{\delta}{\ID_n})^2 \frac{1}{C_{\Delta}}\Big)r(\frac{\delta}{\ID_n})^2 (\ID_n^{\infty})^2\Big]\\
        &\le& \frac{2c}{ \psi(\delta)^2 \delta^2}\cdot \Big[ \delta^2 + \Big[q^{*}\Big(r(\frac{\delta}{\ID_n})\Big)r(\frac{\delta}{\ID_n})\Big]^2 q^{*}\big(C_{\Delta}^{-1}C_{\beta}^{-2}\big)(\ID_n^{\infty})^2\Big]\\
        &\le& \frac{2c}{ \psi(\delta)^2 \delta ^2}\cdot \Big[ \delta^2 + \delta^2 q^{*}\big(C_{\Delta}^{-1}C_{\beta}^{-2}\big)(\frac{\ID_n^{\infty}}{\ID_n})^2\Big]\\
        &\le& \frac{2c(1+q^{*}\big(C_{\Delta}^{-1}C_{\beta}^{-2}\big)(\frac{\ID_n^{\infty}}{\ID_n})^2)}{\psi(\delta)^2},
    \end{eqnarray*}
    which shows \reff{lemma_chaining2_res3}.
\end{proof}

\begin{proof}[Proof of Theorem \ref{thm_martingale_equicont}]
	In the following, we abbreviate $\IH(\delta) = \IH(\delta,\sF, V)$ and $\N(\delta) = \N(\delta,\sF, V)$. The proof follows the lines of \cite[Theorem 4.4]{empproc}. We present it here for completeness. Recall again that for $m > 0$, $\varphi_m^{\wedge}:\R \to \R$ and the corresponding ``peaky'' residual function $\varphi_m^{\vee}:\R \to \R$ via  
\[
    \varphi_{m}^{\wedge}(x) := (x\vee (-m))\wedge m, \quad\quad 
    \varphi_m^{\vee}(x) := x - \varphi_m^{\wedge}(x).
\]
	
	We choose $\delta_0 = \sigma$ and $\delta_j = 2^{-j}\delta_0$, and
    \[
        m_j = \frac{1}{2}m(n,\delta_j,N_{j+1}),
    \]
    as well as $M_n = \frac{1}{2} m_0$. We then use
    \begin{equation}
        \E \sup_{f\in \sF}\Big|\G_n^{(1)}(f)\Big| \le \E \sup_{f\in \sF(M_n)}\Big|\G_n^{(1)}(f)\Big| + \frac{1}{\sqrt{n}}\sum_{i=1}^{n}\E\big[F(Z_i)\Ii_{\{F(Z_i) > M_n\}}\big],\label{thm_martingale_equicont_part1}
    \end{equation}
    where $\sF(M_n) := \{\varphi_{M_n}^{\wedge}(f):f\in \sF\}$.
    
   We construct a nested sequence of partitions $(\sF_{jk})_{k=1,...,N_j}$, $j\in\N$ of $\sF(M_n)$ (where $N_j := \N(\delta_0)\cdot ... \cdot \N(\delta_j)$), and a sequence $\Delta_{jk}$ of measurable functions such that
    \[
        \sup_{f,g\in \sF_{jk}}|f-g| \le \Delta_{jk}, \quad\quad V(\Delta_{jk}) \le \delta_j.
    \]
   In each $\sF_{jk}$, we fix some $f_{jk} \in \sF$, and define $\pi_{j}f := f_{j,\psi_j f}$ where $\psi_{j}f := \min\{i \in \{1,...,N_j\}: f \in \sF_{ji}\}$, and put $\Delta_{j}f := \Delta_{j,\psi_j f}$, and
    \[
        I(\sigma) := \int_0^{\sigma}\psi(\varepsilon)\sqrt{1 \vee \IH(\varepsilon,\sF, V)} d \varepsilon,
    \]
    as well as
    \begin{equation}
        \tau := \min\Big\{j \ge 0: \delta_j \le \frac{I(\sigma)}{\sqrt{n}}\Big\} \vee 1.\label{definition_stoppingtime_chaining_2}
    \end{equation}
    
    For functions $f,g$ with $|f| \le g$, it holds that
	\begin{eqnarray*}
	    |\G_n^{(1)}(f)| &\le& |\G_n^{(1)}(g)| + 2\sqrt{n}\cdot \frac{1}{n}\sum_{i=1}^{n}\IE[g(Z_i,\frac{i}{n})|Z_{i-1}]\\
	    &\le& |\G_n^{(1)}(g)| + 2|\G_n^{(2)}(g)| + 2\sqrt{n}\cdot \frac{1}{n}\sum_{i=1}^{n}\IE[g(Z_i,\frac{i}{n})]\\
	    &\le& |\G_n^{(1)}(g)| + 2|\G_n^{(2)}(g)| + 2\sqrt{n}\|g\|_{1,n}.
	\end{eqnarray*}
	
	Using a similar approach as in \cite[Section 7.2, equations (7.8) and (7.9)]{empproc} applied to $W_i(f) = f(Z_i,\frac{i}{n}) - \E[f(Z_i,\frac{i}{n})|Z_{i-1}]$, and the fact that $\|f - \pi_0 f\|_{\infty} \le 2M_n \le m_0$, we have the decomposition
\begin{eqnarray}
    \sup_{f\in\sF}|\G_n^{(1)}(f)| &\le& \sup_{f\in\sF}|\G_n^{(1)}(\pi_0 f)|\nonumber\\
    &&\quad\quad + \sup_{f\in \sF}|\G_n^{(1)}(\varphi_{m_\tau}^{\wedge}(f-\pi_\tau f))| + \sum_{j=0}^{\tau-1}\sup_{f\in\sF}\Big|\G_n^{(1)}(\varphi_{m_j-m_{j+1}}^{\wedge}(\pi_{j+1}f - \pi_j f))\Big|\nonumber\\
    &&\quad\quad\quad\quad+ \sum_{j=0}^{\tau-1}\sup_{f\in\sF}|\G_n^{(1)}(R(j))|\nonumber\\
    &\le& \sup_{f\in\sF}|\G_n^{(1)}(\pi_0 f)|\nonumber\\
    &&\quad\quad + \Big\{\sup_{f\in\sF}|\G_n^{(1)}(\varphi_{m_\tau}^{\wedge}(\Delta_\tau f))| + 2\sup_{f\in\sF}|\G_n^{(2)}(\varphi_{m_\tau}^{\wedge}(\Delta_\tau f))|\nonumber\\
    &&\quad\quad\quad\quad\quad\quad+ 2\sqrt{n}\sup_{f\in \sF}\|\Delta_\tau f\|_{1,n}\Big\}\nonumber\\
    &&\quad\quad + \sum_{j=0}^{\tau-1}\sup_{f\in\sF}\Big|\G_n^{(1)}(\varphi_{m_j-m_{j+1}}^{\wedge}(\pi_{j+1}f - \pi_j f))\Big|\nonumber\\
    &&\quad\quad + \sum_{j=0}^{\tau-1}\Big\{\sup_{f\in\sF}\Big|\G_n^{(1)}(\min\big\{\big|\varphi_{m_{j+1}}^{\vee}(\Delta_{j+1}f)\big|,2m_j\big\})\Big|\nonumber\\
    &&\quad\quad\quad\quad\quad\quad\quad\quad\quad\quad + 2\sup_{f\in\sF}\Big|\G_n^{(2)}(\min\big\{\big|\varphi_{m_{j+1}}^{\vee}(\Delta_{j+1}f)\big|,2m_j\big\})\Big|\nonumber\\
    &&\quad\quad\quad\quad\quad\quad\quad\quad\quad\quad+ 2\sqrt{n}\sup_{f\in\sF}\| \Delta_{j+1}f \Ii_{\{\Delta_{j+1}f > m_{j+1}\}}\|_{1,n}\Big\}\nonumber\\
    &&\quad\quad + \sum_{j=0}^{\tau-1}\Big\{\sup_{f\in\sF}\Big|\G_n^{(1)}(\min\big\{\big|\varphi_{m_j-m_{j+1}}^{\vee}(\Delta_{j}f)\big|,2m_j\big\})\Big|\nonumber\\
    &&\quad\quad\quad\quad\quad\quad\quad\quad\quad\quad + 2\sup_{f\in\sF}\Big|\G_n^{(2)}(\min\big\{\big|\varphi_{m_j-m_{j+1}}^{\vee}(\Delta_{j}f)\big|,2m_j\big\})\Big|\nonumber\\
    &&\quad\quad\quad\quad\quad\quad\quad\quad\quad\quad + 2\sqrt{n}\sup_{f\in\sF}\| \Delta_{j}f \Ii_{\{\Delta_{j}f > m_j - m_{j+1}\}}\|_{1,n}\Big\}\label{truncation_decomposition_final_martingale}
\end{eqnarray}

    We have for $f\in \sF(M_n)$,
    \begin{eqnarray}
        \pi_0 f &=& \varphi^{\wedge}_{2M_n}(\pi_0 f),\nonumber\\
        \varphi^{\wedge}_{m_{\tau}}(\Delta_{\tau}f) &\le& \min\{\Delta_{\tau}f, 2m_{\tau}\},\nonumber\\
        \varphi^{\wedge}_{m_j - m_{j-1}}(\pi_{j+1}f - \pi_j f) &\le& \min\{\Delta_{j}f, 2m_j\},\nonumber\\
        \min\{\varphi_{m_{j+1}}^{\vee}(\Delta_{j+1}f), 2m_j\}&\le& \min\{\Delta_{j}f, 2m_j\},\nonumber\\
         \min\{\varphi_{m_{j} - m_{j+1}}^{\vee}(\Delta_{j}f), 2m_j\}&\le& \min\{\Delta_{j}f, 2m_j\}.\label{lemma_martingale_chaining_eq10}
    \end{eqnarray}
    We therefore define the event
	\begin{eqnarray*}
		\Omega_n &:=& \{\sup_{f\in \sF(M_n)}R_n(\varphi^{\wedge}_{2M_n}(\pi_0 f))  \le 2\sigma \psi(\sigma)\}\\
		&&\quad\quad \cap \bigcap_{j=1}^{\tau}\big\{ \sup_{f\in \sF(M_n)}R_n(\min\{\Delta_{j}f,2m_j\}) \le 2 \delta_j\psi(\delta_j)\big\}.
	\end{eqnarray*}
	
	From \reff{truncation_decomposition_final_martingale} and \reff{lemma_martingale_chaining_eq10}, we obtain
	\begin{eqnarray}
		&&\sup_{f\in \sF(M_n)}|\G_n^{(1)}(f)|\Ii_{\Omega_n}\nonumber\\
		&\le& \sup_{f\in \sF(M_n)}|\G_n^{(1)}(\pi_0f)|\Ii_{\{\sup_{f\in \sF(M_n)}R_n(\pi_0 f) \le 2\sigma\psi(\sigma)\}}\nonumber\\
    &&\quad\quad + \Big\{\sup_{f\in\sF}|\G_n^{(1)}(\varphi_{m_\tau}^{\wedge}(\Delta_\tau f))|\nonumber\\
    &&\quad\quad\quad\quad\quad\quad\quad\quad \times\Ii_{\{\sup_{f\in \sF(M_n)}R_n(\min\{\Delta_{\tau} f, 2m_{\tau}\}) \le 2\delta_{\tau}\psi(\delta_{\tau})\}} + 2 R_2\Big\}\nonumber\\
    &&\quad\quad + \sum_{j=0}^{\tau-1}\sup_{f\in\sF}\Big|\G_n^{(1)}(\varphi_{m_j-m_{j+1}}^{\wedge}(\pi_{j+1}f - \pi_j f))\Big|\nonumber\\
    &&\quad\quad\quad\quad\quad\quad\quad\quad \times \Ii_{\{\sup_{f\in \sF(M_n)}R_n(\min\{\Delta_j f,2m_j\}) \le 2\delta_j\psi(\delta_j)\}}\nonumber\\
    &&\quad\quad + \sum_{j=0}^{\tau-1}\sup_{f\in\sF}\Big|\G_n^{(1)}(\min\big\{\big|\varphi_{m_{j+1}}^{\vee}(\Delta_{j+1}f)\big|,2m_j\big\})\Big|\nonumber\\
    &&\quad\quad\quad\quad\quad\quad\quad\quad \times\Ii_{\{\sup_{f\in \sF(M_n)}R_n(\min\{\Delta_j f, 2m_j\}) \le 2\delta_j\psi(\delta_j)\}} + 2 R_4\nonumber\\
    &&\quad\quad + \sum_{j=0}^{\tau-1}\sup_{f\in\sF}\Big|\G_n^{(1)}(\min\big\{\big|\varphi_{m_j-m_{j+1}}^{\vee}(\Delta_{j}f)\big|,2m_j\big\})\Big|\nonumber\\
    &&\quad\quad\quad\quad\quad\quad\quad\quad \times\Ii_{\{\sup_{f\in \sF(M_n)}R_n(\min\{\Delta_j f, 2m_j\}) \le 2\delta_j\psi(\delta_j)\}} + 2 R_5\nonumber\\
    &=:& \tilde R_1 + \{\tilde R_2 + 2 R_2\} + \tilde R_3 + \{\tilde R_4 + 2 R_4\} + \{\tilde R_5 + 2R_5\}.\label{martingale_main_chaining_decomposition}
	\end{eqnarray}
		
	We now discuss the terms $\tilde R_i$, $i=1,...,5$ separately. The terms $R_i$, $i \in \{2,4,5\}$ can be discussed similarly to the proof found in \cite[Theorem 4.4]{empproc}. Put
	\[
	    \tilde C_n := 2c(1+\frac{\ID_n^{\infty}}{\ID_n}),
	\]
	where $c$ is a resulting constant from the bound in \cite[Theorem 4.1 or Lemma 7.2]{empproc}.
	
	\begin{itemize}
        \item Since $|\{\pi_0 f: f\in \sF(M_n)\}| \le \N(\delta_0)$, $\|\pi_0 f\|_{\infty} \le M_n \le m(n,\delta_0,\N(\delta_1))$, we have by Lemma \ref{lemma_chaining2}:
    \[
        \IE \tilde R_1 = \E \sup_{f\in \sF(M_n)}|\G_n^{(1)}(\pi_0f)|\Ii_{\{\sup_{f\in \sF(M_n)}R_n(\pi_0 f) \le 2\delta_0\psi(\delta_0)\}} \le \tilde C_n\psi(\delta_0)\delta_0 \sqrt{1 \vee \log \N(\delta_1)}.
    \]
    \item It holds that $|\{\varphi^{\wedge}_{m_{\tau}}(\Delta_{\tau} f): f\in \sF(M_n)\}| \le N_{\tau}$. If $g := \varphi^{\wedge}_{m_{\tau}}(\Delta_{\tau} f)$, then $\|g\|_{\infty} \le m_{\tau} \le  m(n,\delta_{\tau},N_{\tau+1})$. We conclude by Lemma \ref{lemma_chaining2}:
    \begin{eqnarray*}
        \IE \tilde R_2&\le&\IE \sup_{f\in\sF}|\G_n^{(1)}(\varphi_{m_\tau}^{\wedge}(\Delta_\tau f))|\\
    &&\quad\quad\quad\quad\quad\quad\quad\quad \times\Ii_{\{\sup_{f\in \sF(M_n)}R_n(\min\{\Delta_{\tau} f, 2m_{\tau}\}) \le 2\delta_{\tau}\psi(\delta_{\tau})\}}\\
    &\le& \tilde C_n \psi(\delta_{\tau})\delta_{\tau}\cdot \sqrt{1 \vee \log N_{\tau+1}}.
    \end{eqnarray*}
    \item Since the partitions are nested, it holds that $|\{\varphi^{\wedge}_{m_j-m_{j+1}}(\pi_{j+1}f - \pi_j f): f\in \sF(M_n)\}| \le N_{j+1}$. If $g := \varphi^{\wedge}_{m_j-m_{j+1}}(\pi_{j+1}f - \pi_j f)$, we have  $\|g\|_{\infty} \le m_j - m_{j+1} \le m_j \le m(n,\delta_j,N_{j+1})$. We conclude by Lemma \ref{lemma_chaining2}:
    \begin{eqnarray*}
        \IE \tilde R_3 &\le& \sum_{j=0}^{\tau-1}\IE \sup_{f\in\sF}\Big|\G_n^{(1)}(\varphi_{m_j-m_{j+1}}^{\wedge}(\pi_{j+1}f - \pi_j f))\Big|\\
    &&\quad\quad\quad\quad\quad\quad\quad\quad \times \Ii_{\{\sup_{f\in \sF(M_n)}R_n(\min\{\Delta_j f,2m_j\}) \le 2\delta_j\psi(\delta_j)\}}\\
    &\le& \tilde C_n \sum_{j=0}^{\tau-1}\psi(\delta_j)\delta_j \sqrt{1 \vee \log N_{j+1}}.
    \end{eqnarray*}
    \item It holds that $|\{\min\{\varphi^{\vee}_{m_{j+1}}(\Delta_{j+1}f),2m_j\}: f\in \sF(M_n)\}| \le N_{j+1}$. If $g := \min\{\varphi^{\vee}_{m_{j+1}}(\Delta_{j+1}f),2m_j\}$, we have  $\|g\|_{\infty} \le 2m_j = m(n,\delta_j,N_{j+1})$. We conclude by Lemma \ref{lemma_chaining2}:
    \begin{eqnarray*}
        \IE \tilde R_4 &\le& \sum_{j=0}^{\tau-1}\IE \sup_{f\in\sF}\Big|\G_n^{(1)}(\min\big\{\big|\varphi_{m_{j+1}}^{\vee}(\Delta_{j+1}f)\big|,2m_j\big\})\Big|\\
    &&\quad\quad\quad\quad\quad\quad\quad\quad \times\Ii_{\{\sup_{f\in \sF(M_n)}R_n(\min\{\Delta_j f, 2m_j\}) \le 2\delta_j\psi(\delta_j)\}}\\
        &\le& \tilde C_n \sum_{j=0}^{\tau-1}\psi(\delta_j)\delta_j \sqrt{1 \vee \log N_{j+1}}.
    \end{eqnarray*}
    \item It holds that $|\{\min\{\varphi^{\vee}_{m_j - m_{j+1}}(\Delta_{j}f),2m_j\}: f\in \sF(M_n)\}| \le N_{j+1}$. If $g := \min\{\varphi^{\vee}_{m_j - m_{j+1}}(\Delta_{j}f),2m_j\}$, we have  $\|g\|_{\infty} \le 2m_j = m(n,\delta_j,N_{j+1})$. We conclude by Lemma \ref{lemma_chaining2} that:
    \begin{eqnarray*}
        \IE \tilde R_5 &\le& \sum_{j=0}^{\tau-1}\IE \sup_{f\in\sF}\Big|\G_n^{(1)}(\min\big\{\big|\varphi_{m_j-m_{j+1}}^{\vee}(\Delta_{j}f)\big|,2m_j\big\})\Big|\\
    &&\quad\quad\quad\quad\quad\quad\quad\quad \times\Ii_{\{\sup_{f\in \sF(M_n)}R_n(\min\{\Delta_j f, 2m_j\}) \le 2\delta_j\psi(\delta_j)\}}\\
        &\le& \tilde C_n \sum_{j=0}^{\tau-1}\psi(\delta_j)\delta_j \cdot \sqrt{1 \vee \log N_{j+1}}.
    \end{eqnarray*}
    \end{itemize}

 Inserting the bounds for $\IE \tilde R_i$, $i = 1,...,5$ and the bounds for $R_i$, $i \in \{2,4,5\}$ from the proof of \cite[Theorem 4.4]{empproc} into \reff{martingale_main_chaining_decomposition}, we obtain that with some universal constant $\tilde c > 0$, 
    \begin{equation}
        \E \sup_{f\in \sF(M_n)}\Big|\G_n^{(1)}(f)\Big|\Ii_{\Omega_n} \le \tilde c (1 + \frac{\ID_n^{\infty}}{\ID_n} + \frac{\ID_n}{\ID_n^{\infty}})\Big[\sum_{j=0}^{\tau+1}\psi(\delta_j)\delta_j \sqrt{1 \vee \log N_{j+1}} + I(\sigma)\Big].\label{proof_chaining_martingales_eq7}
    \end{equation}

Note that
\[
    \sum_{j=k}^{\infty}\delta_j \psi(\delta_j) \le 2\sum_{j=k}^{\infty}\int_{\delta_{j+1}}^{\delta_j}\psi(\delta_j) dx \le 2\int_0^{\delta_k}\psi(x) dx.
\]
By partial integration, it is easy to see that there exists some universal constant $c_{\psi} >0$ such that
\begin{equation}
    \big|\int_0^{\delta_k} \psi(x) dx\big| \le c_{\psi} \delta_k\psi(\delta_k),\label{psi_partialintegration}
\end{equation}
thus
\begin{equation}
    \sum_{j=k}^{\infty}\delta_j \psi(\delta_j) \le 2c_{\psi} \delta_k\psi(\delta_k).\label{proof_chaining_martingales_eq8}
\end{equation}
Using \reff{proof_chaining_martingales_eq8}, we can argue as in the proof \cite[Theorem 4.4]{empproc} (see (7.44), (7.45) and (7.46) therein) that there exists some universal constant $\tilde c_2 > 0$ such that
\[
    \sum_{j=0}^{\infty}\psi(\delta_j) \delta_j \sqrt{1 \vee \log N_{j+1}} \le \tilde c_2 I(\sigma).
\]
Insertion of the results into \reff{proof_chaining_martingales_eq7} yields
\begin{equation}
	\E \sup_{f\in \sF(M_n)}\big|\G_n^{(1)}(f)\big|\Ii_{\Omega_n} \le \tilde c \cdot (3\tilde c_2+1) (1 + \frac{\ID_n^{\infty}}{\ID_n} + \frac{\ID_n}{\ID_n^{\infty}})  I(\sigma).\label{thm_martingale_equicont_part2}
\end{equation}

\underline{Discussion of the event $\Omega_n$:} We have
\begin{eqnarray}
	\IP(\Omega_n^c) &\le& \IP\Big(\sup_{f\in \sF(M_n)}R_n(\varphi_{2M_n}^{\wedge}(\pi_0 f)) > 2\psi(\sigma)\sigma\Big)\nonumber\\
	&&\quad\quad + \sum_{j=1}^{\tau+1}\IP\Big(\sup_{f\in \sF(M_n)}R_n(\min\{\Delta_{j}f,2m_j\}) > 2 \psi(\delta_j)\delta_j\Big)\nonumber\\
	&=:& R_1^{\circ} + R_2^{\circ}.\label{thm_martingale_equicont_eq0}
\end{eqnarray}

We now discuss $R_i^{\circ}$, $i = 1,2$. Put
\[
    C_n^{\circ} := 2c\Big\{1+q^{*}\big(C_{\Delta}^{-1}C_{\beta}^{-2}\big)\big(\frac{\ID_n^{\infty}}{\ID_n}\big)^2\Big\},
\]
where $c$ is from Lemma \ref{lemma_chaining2}.

\begin{itemize}
    \item Since $|\{\varphi_{2M_n}^{\wedge}(\pi_0 f): f\in \sF(M_n)\}| \le \N(\delta_0) = \N(\sigma)$, $\|\varphi_{2M_n}^{\wedge}(\pi_0 f)\|_{\infty} \le 2M_n \le m(n,\sigma,\N(\sigma))$ and $V(\varphi_{2M_n}^{\wedge}(\pi_0 f)) \le V(\pi_0 f) \le \sigma$, we have by Lemma \ref{lemma_chaining2}:
    \[
        R_1^{\circ} \le \frac{C_n^{\circ}}{\psi(\sigma)^2}.
    \]
    \item It holds that $|\{\min\{\Delta_{j}f,2m_j\}: f\in \sF(M_n)\}| \le N_{j+1}$. We have  $\|\min\{\Delta_{j}f,2m_j\}\|_{\infty} \le 2m_j = m(n,\delta_j,N_{j+1})$ and $V(\min\{\Delta_{j}f,2m_j\}) \le V(\Delta_j f) \le \delta_j$. We conclude by Lemma \ref{lemma_chaining2} that:
    \[
        R_3^{\circ} \le C_n^{\circ}\sum_{j=0}^{\tau+1} \frac{1}{\psi(\delta_j)^2}.
    \]
\end{itemize}

Inserting the bounds for $R_i^{\circ}$, $i = 1,2$, into \reff{thm_martingale_equicont_eq0} yields
\begin{equation}
     \IP(\Omega_n^{c}) \le 2 C_n^{\circ}\sum_{j=0}^{\infty}\frac{1}{\psi(\delta_j)^2}.\label{thm_martingale_equicont_part3}
\end{equation}
We now have
    \[
    \sum_{j=0}^{\infty}\frac{1}{\psi(\delta_j)^2} \le 2\int_0^{\sigma}\frac{1}{\varepsilon\psi(\varepsilon)^2} d\varepsilon = \frac{2}{\log(\log(\sigma))}.
    \]
We conclude that for each $\eta > 0$, 
\begin{eqnarray*}
    \IP\Big(\sup_{f\in \sF}|\G_n^{(1)}(f)| > \eta\Big) &\le&\IP\Big(\sup_{f\in \sF}|\G_n^{(1)}(f)| > \eta, \Omega_n\Big) + \IP(\Omega_n^c)\\
    &\le& \frac{1}{\eta}\E \sup_{f\in \sF}|\G_n^{(1)}(f)|\Ii_{\Omega_n} + \IP(\Omega_n^c).
\end{eqnarray*}
    Insertion of \reff{thm_martingale_equicont_part1}, \reff{thm_martingale_equicont_part2} and \reff{thm_martingale_equicont_part3} gives the result.
\end{proof}

\begin{proof}[Proof of Corollary \ref{cor_martingale_equicont}] We will follow the proof of \cite[Corollary 4.5]{empproc}. Define $\tilde \sF _:= \{f-g : f,g \in \sF\}$. We obtain
    \begin{eqnarray}
        &&\IP\Big( \sup_{V(f-g) \leq \sigma, \, f,g \in \sF} |\G_n(f) - \G_n(g)| \geq \eta \Big)\nonumber\\
        &\le& \IP\Big( \sup_{V(\tilde f) \leq \sigma, \, \tilde f\in \tilde \sF} |\G_n^{(1)}(\tilde f)| \geq \frac{\eta}{2} \Big) + \IP\Big( \sup_{V(\tilde f) \leq \sigma, \, \tilde f \in \tilde \sF} |\G_n^{(2)}(\tilde f)| \geq \frac{\eta}{2} \Big).\label{cor_martingale_equicont_decomp}
    \end{eqnarray}
    
    Now let $F(z,u) := 2D_{n}^{\infty}(u)\cdot \bar F(z,u)$, where $\bar F$ is from Assumption \ref{ass_clt_expansion_ass2}. Then obviously, $F$ is an envelope function of $\tilde \sF$.
    
    We now discuss the second summand on the right hand side in  \reff{cor_martingale_equicont_decomp}. By Markov's inequality and \cite[Theorem 4.4]{empproc} applied to $W_i(f) = \IE[f(Z_i,\frac{i}{n})|Z_{i-1}]$, we obtain as in the proof of \cite[Corollary 4.5]{empproc} that
    \begin{eqnarray}
        &&\IP\Big( \sup_{V(\tilde f) \leq \sigma, \, \tilde f \in \tilde \sF} |\G_n^{(2)}(\tilde f)| \geq \frac{\eta}{2} \Big)\nonumber\\
        &\le& \frac{\tilde c }{(\eta/2)}\Big[2\sqrt{2}(1+\frac{\ID_n^{\infty}}{\ID_n} + \frac{\ID_n}{\ID_n^{\infty}})\int_0^{\sigma/2}\sqrt{1 \vee \IH(u,\sF,V)} d u\nonumber\\
        &&\quad\quad\quad\quad + \frac{4\sqrt{1\vee \IH(\frac{\sigma}{2})}}{r(\frac{\sigma}{\ID_n})}\big\|F^2 \Ii_{\{F > \frac{1}{4}n^{1/2}\frac{r(\sigma)}{\sqrt{1\vee \IH(\frac{\sigma}{2})}}\}}\big\|_{1,n}\Big].\label{cor_martingale_equicont_decomp2}
    \end{eqnarray}
    The first summand in \reff{cor_martingale_equicont_decomp2} converges to $0$ for $\sigma \to 0$ (uniformly in $n$) since
    \[
        \sup_{n\in\N}\int_0^{\sigma/2}\sqrt{1 \vee \IH(u,\sF,V)} du \le \sup_{n\in\N}\int_0^{\sigma} \psi(\varepsilon)\sqrt{1 \vee \IH(\varepsilon,\sF,V)} d\varepsilon < \infty.
    \]
    We now discuss the second summand in \reff{cor_martingale_equicont_decomp2}. The continuity conditions from Assumption \ref{ass_clt_expansion_ass2} on $\bar F$ yield as in the proof of Lemma \ref{lemma_theorem_clt_mult}(ii) that for all $u,u_1,u_2,v_1,v_2 \in [0,1]$,
    \begin{eqnarray}
        \|\bar F(Z_i,u) - \bar F(\tilde Z_i(\frac{i}{n}),u)\|_2 \le C_{cont}\cdot n^{-\alpha s/2},\label{cor_martingale_equicont_eq4}\\
        \|\bar F(Z_i(v_1),u_1) - \bar F(\tilde Z_i(v_2),v_2)\|_2 \le C_{cont}\cdot\big(|v_1 - v_2|^{\alpha s/2} + |u_1 - u_2|^{\alpha s}\big).\label{cor_martingale_equicont_eq5}
    \end{eqnarray}
    In the same manner of \cite[Corollary 4.5]{empproc}, we now obtain with \reff{cor_martingale_equicont_eq4} and \reff{cor_martingale_equicont_eq5} that
    \begin{equation}
        \big\|F^2 \Ii_{\{F > \frac{1}{4}n^{1/2}\frac{r(\sigma)}{\sqrt{1 \vee \IH(\frac{\sigma}{2})}}\}}\big\|_{1,n} \to 0\label{cor_martingale_equicont_eq7}
    \end{equation}
    for $n\to\infty$ (this is obvious if $Z_i$ is stationary, i.e. the first part of Assumption \ref{ass_clt_expansion_ass2} is fulfilled), which shows that \reff{cor_martingale_equicont_decomp2} converges to $0$ for $\sigma \to 0$, $n\to\infty$.
    
    We now consider the first term in \reff{cor_martingale_equicont_decomp}. By  Theorem \ref{thm_martingale_equicont}, we have with some universal constant $c > 0$ that 
    \begin{eqnarray}
        &&\IP\Big( \sup_{V(\tilde f) \leq \sigma, \, \tilde f\in \tilde \sF} |\G_n^{(1)}(\tilde f)| \geq \frac{\eta}{2} \Big)\nonumber\\
        &\le& \frac{2}{\eta}\Big[c\Big(1 + \frac{\ID_n^{\infty}}{\ID_n} + \frac{\ID_n}{\ID_n^{\infty}}\Big)\cdot \int_0^{\sigma} \psi(\varepsilon) \sqrt{1 \vee \IH\big(\eps,\tilde\sF,V\big)} \, \mathrm{d}\eps\nonumber\\
        &&\quad\quad + \frac{4\sqrt{1 \vee \IH(\frac{\sigma}{2})}}{r(\frac{\sigma}{\ID_n})}\big\|F^2\Ii_{\{F > \frac{1}{4}m(n,\sigma,\N(\frac{\sigma}{2}))\}}\big\|_1\Big]\nonumber\\
         &&\quad\quad +c\Big(1 + q^{*}\big(C_{\Delta}^{-1}C_{\beta}^{-2}\big)\Big(\frac{\ID_n^{\infty}}{\ID_n}\Big)^2\Big)\int_0^{\sigma}\frac{1}{\varepsilon \psi(\varepsilon)^2}d\varepsilon.\label{cor_martingale_equicont_decomp3}
    \end{eqnarray}
    For the first summand in \reff{cor_martingale_equicont_decomp3},
    \begin{eqnarray*}
        &&\int_0^{\sigma}\psi(\varepsilon)\sqrt{1 \vee \IH(\varepsilon,\tilde \sF, V)}d\varepsilon\\
        &\le& 2\sqrt{2}\int_0^{\sigma/2}\psi(2\varepsilon)\sqrt{1 \vee \IH(\varepsilon,\sF, V)}d\varepsilon \le 2\sqrt{2}\int_0^{\sigma/2}\psi(\varepsilon)\sqrt{1 \vee \IH(\varepsilon,\sF, V)}d\varepsilon.
    \end{eqnarray*}
    Note that it is easily seen that $\N(\varepsilon,\tilde \sF, V) \le \N(\frac{\varepsilon}{2},\sF, V)^2$ (cf. \cite{Vaart98}, Theorem 19.5), thus
    \begin{equation}
        \IH(\varepsilon,\tilde \sF, V) \le 2 \IH(\frac{\varepsilon}{2},\sF, V)\label{cor_equicont_eq1}.
    \end{equation}
    Together with \reff{cor_martingale_equicont_eq0} and the  uniform boundedness of $\ID_n, \ID_n^{\infty}$, we obtain that the first summand in \reff{cor_martingale_equicont_decomp3} converges to $0$ for $\sigma \to 0$ (uniformly in $n$).
    
    The third summand in \reff{cor_martingale_equicont_decomp3} converges to $0$ for $\sigma \to 0$ (uniformly in $n$) since $\int_0^{\infty}\varepsilon \psi(\varepsilon)^2 d\varepsilon < \infty$ and by the uniform boundedness of $\ID_n, \ID_n^{\infty}$.
    
    The second summand in \reff{cor_martingale_equicont_decomp3} converges to $0$ for $n\to\infty$ by \reff{cor_martingale_equicont_eq7}.
\end{proof}

\subsection{Proofs of Section \ref{sec_clt}}
\label{sec_clt_supp}

\begin{lem}\label{lemma_theorem_clt_mult}
    Let $\sF$ satisfy Assumptions \ref{ass_clt_process}, \ref{ass_clt_fcont}. Suppose that Assumptions \ref{ass2}, \ref{ass_clt_expansion_ass2} hold. Then there exist constants $C_{cont} > 0, C_{\bar f} > 0$ such that for any $f\in \sF$,
    \begin{enumerate}
        \item[(i)] for any $j\ge 1$,
        \begin{eqnarray*}
            \|P_{i-j}f(Z_i,u)\|_2 &\le& D_{f,n}(u) \Delta(j),\\
            \sup_{i=1,...,n}\|f(Z_i,u)\|_2 &\le& C_{\Delta}\cdot D_{f,n}(u),\\
            \sup_{i,u}\|\bar f(Z_i,u)\|_2 \le C_{\bar f}, && \sup_{v,u}\|\bar f(\tilde Z_0(v),u)\|_2 \le C_{\bar f}.
        \end{eqnarray*}
        \item[(ii)] with $x = \frac{1}{2}$,
    \begin{eqnarray}
        \|\bar f(Z_i,u) - \bar f(\tilde Z_i(\frac{i}{n}),u)\|_2 &\le& C_{cont}\cdot n^{-\varsigma s x},\label{lemma_theorem_clt_mult_1}\\
        \|\bar f(\tilde Z_i(v_1),u_1) - \bar f(\tilde Z_i(v_2),u_2)\|_2 &\le& C_{cont}\cdot \big(|v_1 - v_2|^{\varsigma s x} + |u_1 - u_2|^{\varsigma s}\big).\label{lemma_theorem_clt_mult_2}
    \end{eqnarray}
    \end{enumerate}
\end{lem}
\begin{proof}[Proof of Lemma \ref{lemma_theorem_clt_mult}]
    \begin{enumerate}
        \item[(i)]
     If Assumption \ref{ass2} is satisfied, we have by Lemma \ref{depend_trans_2} that
     \begin{eqnarray*}
        \|P_{i-j}f(Z_i,u)\|_2 &=& \|P_{i-j}\IE[f(Z_i,u)|\sA_{i-1}]\|_2\\
        &\le& \|\IE[f(Z_i,u)|\sA_{i-1}] - \IE[f(Z_i,u)|\sA_{i-1}]^{*(i-j)}\|_2 \le D_{f,n}(u)\Delta(j).
     \end{eqnarray*}
     The second assertion follows from Lemma \ref{depend_trans_2}.
     
     \item[(ii)] Let $\bar C_R := \sup_{v,u}\|\bar R(\tilde Z_0(v),u)\|_{2}$ and $C_R := \max\{\sup_{i,u}\|R(Z_i,u)\|_{2}, \sup_{u,v}\|R(\tilde Z_0(v),u)\|_{2}\}$. We first use Assumption \ref{ass_clt_fcont} and H\"older's inequality to obtain
     \begin{eqnarray}
         &&\|\bar f(\tilde Z_i(v),u_1) - \bar f(\tilde Z_i(v),u_2)\|_2\\
         &\le& |u_1-u_2|^{\varsigma}\cdot \big(\|\bar R(\tilde Z_i(v),u_1)\|_2 + \|R(\tilde Z_i(v),u_2)\|_2\big)\nonumber\\
         &\le& 2\bar C_R |u_1-u_2|^{\varsigma}.\label{lemma_theorem_clt_mult_contarg2}
    \end{eqnarray}
    Assume w.l.o.g. that
    \[
        \sup_{u,v}\frac{1}{c^s}\IE\Big[\sup_{|a|_{L_{\sF},s} \le c}\big|\bar f(\tilde Z_0(v),u) - \bar f(\tilde Z_0(v)+a,u)\big|^2\Big] \le C_R.
    \] (which is obvious if $Z_i$ is stationary, i.e. the first part of Assumption \ref{ass_clt_expansion_ass2} is fulfilled; in this case $Z_i = \tilde Z_i(v)$ for all $v$).
    Let $c_n > 0$ be some sequence. Let $C_{\bar f} := \max\{\sup_{i,u}\|f(Z_i,u)\|_{2\bar p}, \sup_{u,v}\|f(\tilde Z_0(v),u)\|_{2\bar p}\}$. Then we have by Jensen's inequality,
\begin{eqnarray*}
    && \big\|\bar f(Z_i,u) - \bar f(\tilde Z_i(v),u)\big\|_2\\
    &\le& \IE\Big[ \big|\bar f(Z_i,u) - \bar f(\tilde Z_i(v),u)\big|^2\Ii_{\{|Z_i - \tilde Z_i(v)|_{L_{\sF},s} \le c_{n}\}}\Big]^{1/2}\\
    &&\quad\quad + \IE\Big[(\bar f(Z_i,u) - \bar f(\tilde Z_i(v),u)^2\Ii_{\{|Z_i - \tilde Z_i(v)|_{L_{\sF},s} > c_{n}\}}\Big]^{1/2}\\
    &\le& \IE\Big[ \sup_{|a|_{L_{\sF},s} \le c_{n}}\big|\bar f(\tilde Z_i(v),u) - \bar f(\tilde Z_i(v) + a,u)\big|^2\Big]^{1/2}\\
    &&\quad\quad + \big\{\big\|\bar f(Z_i,u)\big\|_{2\bar p} + \bar f(\tilde Z_i(v),u)\big\|_{2\bar p} \big\}\IP(|Z_i - \tilde Z_i(v)|_{L_{\sF},s} > c_{n})^{\frac{\bar p-1}{2 \bar p}}\\
    &\le& C_R c_n^s + 2C_{\bar f} \Big(\frac{\| |Z_i - \tilde Z_i(v)|_{L_{\sF},s}\|_{\frac{2 \bar p s}{ \bar p-1}}}{c_n}\Big)^s\\
    &\le& C_R c_n^s + 2 C_{\bar f}C_X(|L_{\sF}|_1 + \sum_{j=0}^{\infty}L_{\sF,j}j^{\varsigma s}) \cdot \frac{\{|v-\frac{i}{n}|^{\varsigma s} + n^{-\varsigma s}\}}{c_n^s}.
\end{eqnarray*}
We obtain with $c_{cont} := C_R + 2 C_{\bar f}C_X(|L_{\sF}|_1 + \sum_{j=0}^{\infty}L_{\sF,j}j^{\varsigma s})$ that 
    \begin{equation}
        \|\bar f(Z_{i},u) - \bar f(\tilde Z_i(v),u)\|_2 \le c_{cont}\cdot \Big[c_n^s + \frac{|v-\frac{i}{n}|^{\varsigma s} + n^{-\varsigma s}}{c_n^s}\Big].\label{lemma_theorem_clt_mult_res2}
    \end{equation}
    
    Furthermore, as above, for any $c > 0$, 
    \begin{eqnarray}
        \|f(\tilde Z_i(v_1),u) - f(\tilde Z_i(v_2),u)\|_2 &\le& C_R c^s + 2C_{\bar f}\Big(\frac{\| |\tilde Z_0(v_1) - \tilde Z_0(v_2)|_{L_{\sF},s}^s\|_{\frac{2\bar p}{\bar p-1}}}{c}\Big)^s\nonumber\\
        &\le& C_R c^s + 2 C_{\bar f}C_X |L_{\sF}|_1 \cdot \frac{|v_1 - v_2|^{\varsigma s}}{c^s}.\label{lemma_theorem_clt_mult_res22}
    \end{eqnarray}
    From \reff{lemma_theorem_clt_mult_res2}, we obtain the first assertion with $v = \frac{i}{n}$. The second assertion follows from \reff{lemma_theorem_clt_mult_res22} and \reff{lemma_theorem_clt_mult_contarg2}.
    \end{enumerate}
\end{proof}

\subsection{Details of Section \ref{sec_further_appli}}
\label{sec_examples_supp}

    We first show that the supremum over $x\in\R$, $v\in[0,1]$ can be approximated by a supremum over grids $x\in \sX_n$, $v\in V_n$.
    
    For some $Q > 0$, put $c_n = Qn^{\frac{1}{2s}}$. Define the event $A_n = \{\sup_{i=1,...,n}|X_i| \le c_n\}$. Then by Markov's inequality,
    \begin{equation}
        \IP(A_n^c) \le n \cdot \frac{\|X_i\|_{2s}^{2s}}{Q^{2s}c_n^{2s}} \le \frac{C_X^{2s} n}{c_n^{2s}}\label{example_densityestimation_details_eq1}
    \end{equation}
    is arbitrarily small for $Q$ large enough.
    
    Put $\hat g_{n,h}^{\circ}(x,v) := \frac{1}{n}\sum_{i=1}^{n}K_{h_1}(i/n-v) K_{h_2}(X_i-x)\Ii_{\{|X_i| \le c_n\}}$. Then
    \begin{equation}
        \text{on }A_n, \quad\quad \hat g_{n,h}^{\circ}(\cdot) = \hat g_{n,h}(\cdot).\label{example_densityestimation_details_eq2}
    \end{equation}
    Furthermore,
    \begin{eqnarray}
        \sqrt{nh_1 h_2}\big|\IE \hat g_{n,h}(x,v) - \IE \hat g_{n,h}^{\circ}(x,v)\big| &\le& \frac{\sqrt{n h_1 h_2}|K|_{\infty}}{nh_1}\sum_{i=1}^{n}\IE[ K_{h_2}(X_i - x)\Ii_{\{|X_i| > c_n\}}]\nonumber\\
        &\le& \sqrt{nh_1 h_2}(h_1h_2)^{-1}|K|_{\infty}c_n^{-2s}\sup_{i}\IE[ K(\frac{X_i - x}{h_2})|X_i|^{2s}]\nonumber\\
        &\le& Q^{-2s}(nh_1 h_2)^{-1/2}| K|_{\infty}^2 C_X^{2s} = o(1).\label{example_densityestimation_details_eq3}
    \end{eqnarray}
    For $|x| > 2c_n$, we have $ K_{h_2}(X_i - x)\Ii_{\{|X_i| \le c_n\}} \le h^{-1}(\frac{c_n}{h})^{-p_K} = h^{p_K-1}c_n^{-p_K}$ and thus
    \begin{equation}
        \sqrt{nh}|\hat g_{n,h}^{\circ}(x,v) - \IE \hat g_{n,h}^{\circ}(x,v)| \le \frac{2|K|_{\infty}C_{ K}}{h_1^{1/2}}(nh_2)^{1/2} h_2^{p_K-1}c_n^{-p_K} \le \frac{h_2^{p_K}}{Q^{p_K}(nh_1 h_2)^{1/2}} = o(1).\label{example_densityestimation_details_eq4}
    \end{equation}
    By \reff{example_densityestimation_details_eq2}, \reff{example_densityestimation_details_eq3} and \reff{example_densityestimation_details_eq4}, we have on $A_n$,
    \begin{eqnarray}
        &&\sqrt{nh_1h_2}\sup_{x\in\R, v\in [0,1]}|\hat g_{n,h}(x,v) - \IE \hat g_{n,h}(x,v)|\nonumber\\
        &=& \sqrt{nh_1h_2}\sup_{x\in\R, v \in [0,1]}|\hat g_{n,h}^{\circ}(x,v) - \IE \hat g_{n,h}^{\circ}(x,v)| + o_p(1)\nonumber\\
        &=& \sqrt{nh_1h_2}\sup_{|x|\le 2c_n, v\in[0,1]}|\hat g_{n,h}^{\circ}(x,v) - \IE \hat g_{n,h}^{\circ}(x,v)| + o_p(1)\nonumber\\
        &=& \sqrt{nh_1h_2}\sup_{|x|\le 2c_n, v\in[0,1]}|\hat g_{n,h}(x,v) - \IE \hat g_{n,h}(x,v)| + o_p(1).\label{example_densityestimation_details_eq5}
    \end{eqnarray}
    
    Let $\sX_n = \{i n^{-3}: i \in \{-2 \lceil c_n\rceil n^3 ,..., 2 \lceil c_n\rceil n^3\}\}$ be a grid that approximates each $x \in [-2c_n,2 c_n]$ with precision $n^{-3}$, and $V_n = \{in^{-3}:i=1,...,n^{3}\}$. Since $K$ are Lipschitz continuous with constant $L_K$, 
    \begin{eqnarray}
        &&\sqrt{nh_1 h_2}\sup_{|x-x'| \le n^{-3},|v-v'| \le n^{-3}}\big|\big(\hat g_{n,h}(x,v) - \IE \hat g_{n,h}(x,v)\big)\nonumber\\
        &&\quad\quad\quad\quad\quad\quad\quad\quad\quad\quad\quad\quad- \big(\hat g_{n,h}(x',v) - \IE \hat g_{n,h}(x',v)\big)\big|\nonumber\\
        &\le& 2\frac{\sqrt{n}}{\sqrt{h_1 h_2}}\sup_{|x-x'| \le n^{-3},|v-v'| \le n^{-3}}\Big[\frac{L_{K}|K|_{\infty}|x-x'|}{h_2} + \frac{L_K | K|_{\infty}|v-v'|}{h_1}\Big]\nonumber\\
        &=& O(n^{-1}).\label{example_densityestimation_details_eq6}
    \end{eqnarray}
    We conclude from \reff{example_densityestimation_details_eq1},  \reff{example_densityestimation_details_eq5} and \reff{example_densityestimation_details_eq6} that
    \begin{eqnarray}
        &&\sqrt{nh_1h_2}\sup_{x\in\R, v\in [0,1]}|\hat g_{n,h}(x,v) - \IE \hat g_{n,h}(x,v)|\nonumber\\
        &=& \sqrt{nh_1h_2}\sup_{x\in \sX_n, v\in V_n}|\hat g_{n,h}(x,v) - \IE \hat g_{n,h}(x,v)| + O_p(1)\label{example_densityestimation_details_eq7}
    \end{eqnarray}

    It was already shown that Assumption \ref{ass2} is satisfied. Furthermore, we can choose $\ID_n = |K|_{\infty}$, $\ID_{\nu_2,n}^{\infty} = \frac{|K|_{\infty}}{\sqrt{h_1}}$ with $\nu_2 = \infty$, and $\bar F(z,u) = \sup_{f\in \sF}\bar f(z,u) \le \frac{| K|_{\infty}}{\sqrt{h_2}} =: C_{\bar F,n}$. Note that
    \begin{eqnarray*}
        \IE[(\sqrt{h_2} K_{h_2}(X_i-x))^2] &=& \IE\big[\IE[(\sqrt{h_2} K_{h_2}(X_i-x))^2|X_{i-1}]\big] \\ 
        &=& \int \Big( \int K(w)^\kappa f_{X_i \mid X_{i-1} = z} (x+ wh_2) dw \Big)^{1/\kappa} d \IP^{X_{i-1}}(z) \\
        &\le& C_{\infty}\cdot (\int K(w)^2 dw)^{1/2}.
    \end{eqnarray*}
    therefore
    \[
        \|f_{x,v}\|_{2,n} \le \ID_n C_{\infty}\int K(w)^2 dw,
    \]
    which implies $\sigma := \sup_{n\in\N}\sup_{f\in\sF}V_n(f) < \infty$. Due to $\Delta(k) = O(k^{-\alpha s})$, the last condition in \reff{lemma_hoeffding_dependent_rates_cond1} is fulfilled if
    \[
        \sup_{n\in\N}\frac{\log(n)}{n h_2 h_1^{\frac{\alpha s}{\alpha s-1}}} < \infty.
    \]
    By Corollary \ref{lemma_hoeffding_dependent_rates}, we have
    \[
        \sqrt{nh_1 h_2}\sup_{x\in \sX_n, v \in V_n}\big|\hat g_{n,h}(x) - \IE \hat g_{n,h}(x,v)\big| = \sup_{f\in\sF}|\G_n(f)| = O_p(\sqrt{\log|\sF|}) = O(\sqrt{\log(n)}).
    \]
    With \reff{example_densityestimation_details_eq7}, it follows that
    \[
        \sqrt{nh_1h_2}\sup_{x\in\R, v\in [0,1]}|\hat g_{n,h}(x,v) - \IE \hat g_{n,h}(x,v)| = O_p\big(\sqrt{\log(n)}\big).
    \]

\end{document}